\documentclass[a4paper,10pt,leqno]{amsart}
       \title{$K$- and $L$-theory of group rings over $\GL_n(\IZ)$}
       \author{Arthur Bartels}
       \author{Wolfgang L\"uck}
       \author{Holger Reich}
       \author{Henrik R\"uping}
       \address{Westf\"alische Wilhelms-Universit\"at M\"unster\\
               Mathematisches Institut\\
               Einsteinstr.~62,
               D-48149 M\"unster, Germany}
       \address{Rheinische Wilhelms-Universit\"at Bonn\\
               Mathematisches Institut\\
               Endenicher Allee 62, 53115 Bonn, Germany}
        \address{Freie Universit{\"a}t Berlin\\
               Institut f{\"u}r Mathematik\\
               Arnimallee 7,  14195 Berlin  Germany}
      \address{Rheinische Wilhelms-Universit\"at Bonn\\
               Mathematisches Institut\\
               Endenicher Allee 62, 53115 Bonn, Germany}
      \email{bartelsa@math.uni-muenster.de}
      \urladdr{http://www.math.uni-muenster.de/u/bartelsa}
        \email{wolfgang.lueck@him.uni-bonn.de}
      \urladdr{http://www.him.uni-bonn.de/lueck/}
        \email{holger.reich@fu-berlin.de}
      \urladdr{http://www.math.fu-berlin.de/groups/top/members/reich.html}
      \email{henrik.rueping@hcm.uni-bonn.de}
         \date{April 2013}

\usepackage{pdfsync}
\usepackage{hyperref}
\usepackage{calc}
\usepackage{enumerate,amssymb}
\usepackage{color}
\usepackage[arrow,curve,matrix,tips,2cell]{xy}
  \SelectTips{eu}{10} \UseTips
  \UseAllTwocells
\usepackage{tikz}
\usepackage{pdfcolmk}

\DeclareMathAlphabet{\matheurm}{U}{eur}{m}{n}
\newcommand{\EGF}[2]{E_{#2}(#1)}
\newcommand{\intgf}[2]{\mbox{$\int_{#1} #2$}}

\DeclareMathOperator{\aut}{aut}

\DeclareMathOperator{\id}{id}

\DeclareMathOperator{\pt}{pt}
\DeclareMathOperator{\pr}{pr}
\DeclareMathOperator{\rk}{rk}

\DeclareMathOperator{\sign}{sign}

\DeclareMathOperator{\Sym}{Sym}

\DeclareMathOperator{\tr}{tr}
\DeclareMathOperator{\vol}{vol}
\DeclareMathOperator{\Wh}{Wh}

\newcommand{\VCyc}{{\mathcal{V}\mathcal{C}\text{yc}}}

  \newcommand{\IN}{\mathbb{N}}

  \newcommand{\IQ}{\mathbb{Q}}
  \newcommand{\IR}{\mathbb{R}}

  \newcommand{\IZ}{\mathbb{Z}}

  \newcommand{\cala}{\mathcal{A}}

  \newcommand{\calf}{\mathcal{F}}

  \newcommand{\calh}{\mathcal{H}}

  \newcommand{\call}{\mathcal{L}}

  \newcommand{\calo}{\mathcal{O}}

  \newcommand{\calu}{\mathcal{U}}
  \newcommand{\calv}{\mathcal{V}}
  
  \newcommand{\calw}{\mathcal{W}}

  \newcommand{\bfK}{{\mathbf K}}
  \newcommand{\bfL}{{\mathbf L}}

\newcounter{commentcounter}

\theoremstyle{plain}
\newtheorem{theorem}{Theorem}[section]

\newtheorem{lemma}[theorem]{Lemma}

\newtheorem{corollary}[theorem]{Corollary}
\newtheorem{proposition}[theorem]{Proposition}

\newtheorem*{theorem*}{Theorem}
\newtheorem*{mtheorem}{Main Theorem}
\newtheorem*{gtheorem}{General Theorem}

\theoremstyle{definition}
\newtheorem{definition}[theorem]{Definition}

\theoremstyle{remark}
\newtheorem{remark}[theorem]{Remark}
\newtheorem{example}[theorem]{Example}

\makeatletter\let\c@equation=\c@theorem\makeatother

\newcommand{\x}{{\times}}

\newcommand{\e}{{\varepsilon}}
\newcommand{\CAT}{{\operatorname{CAT}}}
\newcommand{\FS}{\mathit{FS}}

\newcommand{\ev}{\mathit{ev}}

\newcommand{\SL}{{\mathit{SL}}}
\newcommand{\GL}{{\mathit{GL}}}
\newcommand{\ignore}[1]{}

\begin{document}

\begin{abstract}
  We prove the $K$- and $L$-theoretic Farrell-Jones Conjecture (with
  coefficients in  additive categories) for $\GL_n(\IZ)$.
\end{abstract}

\maketitle

%\newlength{\origlabelwidth} \setlength\origlabelwidth\labelwidth

%====================================================

% \setcounter{section}{-1}

\section*{Introduction}
\label{sec:introduction}

  The Farrell-Jones Conjecture predicts a formula for the $K$- and $L$-theory of
  group rings $R[G]$.
  This formula describes these groups in terms of group homology
  and $K$- and $L$-theory of group rings $RV$, where $V$
  varies over the family $\VCyc$ of virtually cyclic subgroups of $G$.
 
  \begin{mtheorem}
    Both the $K$-theoretic and the $L$-theoretic Farrell-Jones Conjecture 
    (see Definition~\ref{def:K-theoretic_Farrell-Jones_Conjecture}
    and~\ref{def:L-theoretic_Farrell-Jones_Conjecture}) hold for
    $\GL_n(\IZ)$.
  \end{mtheorem}

  We will generalize this theorem in the General Theorem below. 
  In particular it also holds for arithmetic groups 
  defined over number fields, compare Example~\ref{exa:ring_of_integers},
  and extends to the more general version ``with wreath products''.
  
  For cocompact lattices in almost connected Lie groups this result holds
  by Bartels-Farrell-L\"uck~\cite{Bartels-Farrell-Lueck(2011cocomlat)}. 
  The lattice $\GL_n(\IZ)$  has finite covolume but is not cocompact. 
  It is a long standing question whether the Baum-Connes Conjecture 
  holds for $\GL_n(\IZ)$.

  For torsion free discrete subgroups of $\GL_n(\IR)$, or more generally, for
  fundamental groups of A-regular complete connected non-positive curved
  Riemannian manifolds, the Farrell-Jones Conjecture with coefficients in $\IZ$
  has been proven by Farrell-Jones~\cite{Farrell-Jones(1998)}.

%%%%%%%%%%%%%%%%%%%%%%%%%%%%%%%%%%%%%%%%%%%%%%%%%%%%%%%%%%%%%%%%%%%%%

\subsection*{The formulation of the Farrell-Jones Conjecture}

  \begin{definition}[$K$-theoretic FJC]
      \label{def:K-theoretic_Farrell-Jones_Conjecture}
    Let $G$ be a group and let $\calf$ be a family of subgroups.  
    Then $G$ satisfies the \emph{$K$-theoretic Farrell-Jones Conjecture  
     %with coefficients in an additive category 
     with respect to $\calf$} if for any
    additive $G$-category $\cala$ the assembly map
    \begin{eqnarray*}
      &  H_n^G\bigl(\EGF{G}{\calf};\bfK_{\cala}\bigr) \to
      H_n^G\bigl(\pt;\bfK_{\cala}\bigr)
      = K_n\left(\intgf{G}{\cala}\right)
      &
    \end{eqnarray*}
    induced by the projection $\EGF{G}{\calf} \to \pt$ is bijective for all 
    $n \in \IZ$.
    If this map is bijective for all $n \leq 0$ and surjective for $n = 1$,
    then we say $G$ satisfies the \emph{$K$-theoretic Farrell-Jones 
     Conjecture up to dimension  $1$ 
     %with coefficients in an additive category 
     with respect to $\calf$}.
    
    If the family $\calf$ is not mentioned, it is by default the family $\VCyc$ of
    virtually cyclic subgroups.
  \end{definition}

  If one chooses $\cala$ to be (a skeleton of) the category of finitely generated
  free $R$-modules with trivial $G$-action, then
  $K_n\left(\intgf{G}{\cala}\right)$ is just the algebraic $K$-theory $K_n(RG)$ of
  the group ring $RG$.
 
  If $G$ is torsion free, $R$ is a regular ring, and $\calf$ is $\VCyc$, then the
  claim boils down to the more familiar statement that the classical
  assembly map $H_n(BG;\bfK_R) \to K_n(RG)$ from the homology theory associated to
  the (non-connective) algebraic $K$-theory spectrum of $R$ applied to the
  classifying space $BG$ of $G$ to the algebraic $K$-theory of $RG$ is a
  bijection.  
  If we restrict further to the case $R = \IZ$ and $n \le 1$, then
  this implies the vanishing of the Whitehead group $\Wh(G)$ of $G$, of the
  reduced projective class group $\widetilde{K}_0(\IZ G)$, and of all negative
  $K$-groups $K_n(\IZ G)$ for $n \le -1$.

  \begin{definition}[$L$-theoretic FJC] 
    \label{def:L-theoretic_Farrell-Jones_Conjecture}
    Let $G$ be a group and let $\calf$ be a family of subgroups.  Then $G$
    satisfies the \emph{$L$-theoretic Farrell-Jones Conjecture 
    %with coefficients in an additive category 
    with respect to $\calf$} if for any additive
    $G$-category with involution $\cala$ the assembly map
    \begin{eqnarray*}
       & %\asmb^{G,\cala}_n \colon
        H_n^G\bigl(\EGF{G}{\calf};\bfL_{\cala}^{\langle - \infty\rangle}\bigr) \to
        H_n^G\bigl(\pt;\bfL_{\cala}^{\langle - \infty\rangle}\bigr)
        = L_n^{\langle - \infty\rangle}\left(\intgf{G}{\cala}\right)
       &
    \end{eqnarray*}
    induced by the projection $\EGF{G}{\calf} \to \pt$ is bijective for all 
    $n \in \IZ$.

    If the family $\calf$ is not mentioned, it is by default the family 
    $\VCyc$ of virtually cyclic subgroups.
  \end{definition}

  Given a group $G$, a \emph{family of subgroups} $\calf$ is a collection of
  subgroups of $G$ that is closed under conjugation and taking subgroups.
  For the notion of a \emph{classifying space $\EGF{G}{\calf}$ for a family
  $\calf$} we refer for instance to the survey article~\cite{Lueck(2005s)}.

  The natural choice for $\calf$ in the Farrell-Jones Conjecture is the family 
  $\VCyc$ of virtually cyclic subgroups but for inductive arguments it is 
  useful to consider other families as well.

  \begin{remark}[Relevance of the additive categories as coefficients]
    The versions of the Farrell-Jones Conjecture appearing in
    Definition~\ref{def:K-theoretic_Farrell-Jones_Conjecture} and
    Definition~\ref{def:L-theoretic_Farrell-Jones_Conjecture} are formulated and
    analyzed in~\cite{Bartels-Lueck(2009coeff)}, \cite{Bartels-Reich(2007coeff)}.
    They encompass the versions for group rings $RG$ over arbitrary rings $R$,
    where one can built in a twisting into the group ring or treat more generally
    crossed product rings $R \ast G$ and one can allow orientation homomorphisms
    $w \colon G \to \{\pm 1\}$ in the $L$-theory case. 
    Moreover, inheritance properties, e.g., passing to subgroups, 
    finite products, finite free products, and directed colimits, 
    are built in and one does not have to pass to fibered versions anymore.
  \end{remark}

  The original source for the (Fibered) Farrell-Jones Conjecture is the paper by
  Farrell-Jones~\cite[1.6 on p.~257 and~1.7 on p.~262]{Farrell-Jones(1993a)}.
  For more information about the Farrell-Jones Conjecture, its relevance and its
  various applications to prominent conjectures due to Bass, Borel, Kaplansky,
  Novikov and Serre, we refer 
  to~\cite{Bartels-Lueck-Reich(2008appl),Lueck-Reich(2005)}.

  We will often abbreviate Farrell-Jones Conjecture to FJC.
  
%=====================================

  \subsection*{Extension to more general rings and groups}
   \label{subsec:Extension_to_more_general_rings_and_groups}

  We will see that it is not hard to generalize the Main Theorem as follows.

  \begin{gtheorem} \label{the:general_theorem}
    Let $R$ be a ring whose underlying abelian group is finitely generated.  
    Let $G$ be a group which is commensurable to a subgroup of $\GL_n(R)$ 
    for some natural number $n$.

    Then $G$ satisfies both the $K$-theoretic and the $L$-theoretic Farrell-Jones
    Conjecture  with  wreath products~\ref{def:FJwF}.
  \end{gtheorem}

  Two groups $G_1$ and $G_2$ are called \emph{commensurable} if they contain
  subgroups $G_1' \subseteq G_1$ and $G_2' \subseteq G_2$ of finite index such
  that $G_1'$ and $G_2'$ are isomorphic.
  In this case $G_1$ satisfies the FJC with wreath products if and only if
  $G_2$ does, see Remark~\ref{rem:inheritance-FJCwreath}.

  \begin{example}[Ring of integers]\label{exa:ring_of_integers}
    Let $K$ be an algebraic number field and $\calo_K$ be its ring of
    integers. 
    Then $\calo_K$ considered as abelian group is finitely generated
    free (see~\cite[Chapter~I, Proposition~2.10 on p.~12]{Neukirch(1999)}).
    Hence by the General Theorem any group $G$ which is
    commensurable to a subgroup of $\GL_n(\calo_K)$ for some natural number $n$
    satisfies both the $K$-theoretic and $L$-theoretic FJC with  wreath products.  
    This includes in particular arithmetic groups over number fields.
  \end{example}

  %===================================================

  \subsection*{Discussion of the proof}

  The proof of the FJC for $\GL_n(\IZ)$
  will use the transitivity principle~\cite[Theorem~A.10]{Farrell-Jones(1993a)}, 
  that we recall here. 

  \begin{proposition}[Transitivity principle]
    \label{prop:trans-princ}
    Let $\calf \subset \calh$ be families of subgroups of $G$.
    Assume that $G$ satisfies the FJC with respect to $\calh$
    and that each $H \in \calh$ satisfies the FJC with
    respect to $\calf$.

    Then $G$ satisfies the FJC with respect to $\calf$.  
  \end{proposition}
 
  This principle applies to all versions of the FJC discussed above. 
  In this form it can be found for example 
  in~\cite[Theorem~1.11]{Bartels-Farrell-Lueck(2011cocomlat)}.
  
  The main step in proving the FJC for $\GL_n(\IZ)$
  is to prove that $\GL_n(\IZ)$ satisfies the FJC
  with respect to a family $\calf_n$.
  This family is defined at the beginning of 
  Section~\ref{sec:Transfer_Reducibility_of_SL_n(IZ)}.
  This family is larger than $\VCyc$ and contains for example $\GL_k(\IZ)$
  for $k < n$.  
  We can then use induction on $n$ to prove that every group from
  $\calf_n$ satisfies the {FJC}.
  At this point  we also  use the fact that
  virtually poly-cyclic groups satisfy the {FJC}.

  To prove that $\GL_n(\IZ)$ satisfies the FJC with respect to $\calf_n$
  we will apply two results from~\cite{Bartels-Lueck(2012annals), Wegner(2012FJ_CAT0)}.
  Originally these results were used to prove that $\CAT(0)$-groups
  satisfy the {FJC}.
  Checking that they are applicable to $\GL_n(\IZ)$ is more difficult.
  While $\GL_n(\IZ)$ is not a $\CAT(0)$-group, it does act on a 
  $\CAT(0)$-space $X$. 
  This action is proper and isometric, but not cocompact.
  Our main technical step is to show that the flow space associated
  to this $\CAT(0)$-space admits \emph{long $\calf_n$-covers at infinity},
  compare Definition~\ref{def:at-infinity_plus_periods}.
  
  In Section~\ref{sec:The_space_of_inner_products_and_the_volume_function} we
  analyze the $\CAT(0)$-space $X$. 
  On it we introduce, following Grayson~\cite{Grayson(1984)}, certain volume
  functions and analyze them from a metric point of view. 
  These functions will be used to cut off a suitable well-chosen 
  neighborhood of infinity so that
  the $GL_n(\IZ)$-action on the complement is cocompact.
  In Section~\ref{sec:Sublattices_of_Zn} we study sublattices in $\IZ^n$. 
  This will be needed to find the neighborhood of infinity mentioned above.
  Here we prove a crucial estimate in Lemma~\ref{lem:estimate_for_d_W}.
  
  As outlined this proof works best for the $K$-theoretic
  {FJC} up to dimension $1$; this case is contained in 
  Section~\ref{sec:Transfer_Reducibility_of_SL_n(IZ)}.
  The modifications needed for the full $K$-theoretic FJC 
  are discussed in Section~\ref{sec:Strong-Trans-Red-GLnZ}
  and use results of Wegner~\cite{Wegner(2012FJ_CAT0)}. 

  For $L$-theory the induction does not work quite as smoothly.
  The appearance of index $2$ overgroups in the statement 
  of~\cite[Theorem~1.1(ii)]{Bartels-Lueck(2012annals)} force us
  to use a stronger induction hypothesis: we need to assume
  that finite overgroups of $\GL_k(\IZ)$, $k < n$ satisfy
  the FJC.
  (It would be enough to consider overgroups of index $2$,
  but this seems not to simplify the argument.)  
  A good formalism to accommodate this is the FJC with
  wreath products (which implies the FJC).
  In Section~\ref{sec:wreath-product_and_transfer-reducibility}
  we provide the necessary extensions of the results
  from~\cite{Bartels-Lueck(2012annals)} for this version of the {FJC}.
  In Section~\ref{sec:FJ-wreath} we then prove the $L$-theoretic FJC with
  wreath products for $\GL_n(\IZ)$.

  In Section~\ref{sec:Proof_of_the_General_Theorem} we give the proof
  of the General Theorem.

%%%%%%%%%%%%%%%%%%%%%%%%%%%%%%%%%%%%%%%%%%%%%%%%%%%%%%%%%%%%%%%%%%%%%%%%

\subsection*{Acknowledgements}
\label{subsec:Acknowledgements}

  We are grateful to Dan Grayson for fruitful discussions about his
  paper~\cite{Grayson(1984)}. 
  We also thank Enrico Leuzinger for answering related questions.

  The work was financially supported by 
  SFB~878 \emph{Groups, Geometry and Actions} in M\"unster, 
  the HCM (Hausdorff Center for Mathematics) in Bonn, 
  and the Leibniz-Preis of the second author.  
  Parts of the paper were developed during the Trimester Program 
  \emph{Rigidity} at the HIM (Hausdorff
  Research Institute for Mathematics) in Bonn in the fall of 2009.

%%%%%%%%%%%%%%%%%%%%%%%%%%%%%%%%%%%%%%%%%%%%%%%%%%%%%%%%%%%%%%%%%%%%%%%%%
%%%%%%%%%%%%%%% Section 1: The space of inner products and the volume
%%%%%%%%%%%%%%% function %%%%%%%%%%%
%%%%%%%%%%%%%%%%%%%%%%%%%%%%%%%%%%%%%%%%%%%%%%%%%%%%%%%%%%%%%%%%%%%%%%%%%

\typeout{-------- Section 1: The space of inner products and the volume
  function --------------}

\section{The space of inner products and the volume function}
\label{sec:The_space_of_inner_products_and_the_volume_function}

Throughout this section let $V$ be an $n$-dimensional real vector space. Let
$\widetilde{X}(V)$ be the set of all inner products on $V$. We want to examine
the smooth manifold $\widetilde{X}(V)$ and equip it with an $\aut(V)$-invariant
complete Riemannian metric with non-positive sectional curvature.  With respect
to this structure we will examine a certain volume function.  We try to keep all
definitions as intrinsic as possible and then afterward discuss what happens
after choices of extra structures (such as bases).

%%%%%%%%%%%%%%%%%%%%%%%%%%%%%%%%%%%%%%%%%%%%%%%%%%%%%%%%%%%%%%%%%%%%%%%%%

\subsection{The space of inner products}
\label{subsec:The_space_of_inner_products}

We can equip $V$ with the structure of a smooth manifold by requiring that any
linear isomorphism $V \to \IR^n$ is a diffeomorphism with respect to the
standard smooth structure on $\IR^n$. In particular $V$ carries a preferred
structure of a (metrizable) topological space and we can talk about limits of
sequences in $V$. We obtain a canonical trivialization of the tangent bundle
$TV$
\begin{eqnarray}
  \phi_V \colon V \times V & \to & TV
  \label{trivialization_phi_V_of_TV}
\end{eqnarray}
which sends $(x,v)$ to the tangent vector in $T_xV$ represented by the smooth
path $\IR \to V, \; t \mapsto x + t \cdot v$. The inverse sends the tangent
vector in $TV$ represented by a path $w \colon (-\epsilon,\epsilon) \to V$ to
$(w(0),w'(0))$. If $f \colon V \to W$ is a linear map, the following diagram
commutes
\begin{eqnarray*}
  \xymatrix{
    V \times V \ar[r]^{\phi_V}_{\cong} \ar[d]^{f \times f} &  TV \ar[d]^{Tf}
    \\
    W \times W \ar[r]^{\phi_W}_{\cong} & TW
  }
\end{eqnarray*}

Let $\hom(V,V^*)$ be the real vector space of linear maps $V \to V^*$ from $V$
to the dual $V^*$ of $V$.  In the sequel we will always identify $V$ and
$(V^*)^*$ by the canonical isomorphism $V \to (V^*)^*$ which sends $v \in V$ to
the linear map $V^* \to \IR,\; \alpha \mapsto \alpha(v)$. Hence for $s \in
\hom(V,V^*)$ its dual $s^* \colon (V^*)^* = V \to V^*$ belongs to $\hom(V,V^*)$
again. Let $\Sym(V) \subseteq \hom(V,V^*)$ be the subvector space of elements
$s\in \hom(V,V^*)$ satisfying $s^* = s$.  We can identify $\Sym(V)$ with the set
of all bilinear symmetric pairings $V \times V \to \IR$, namely, given $s
\in\Sym(V)$ we obtain such a pairing by $(v,w) \mapsto s(v)(w)$.  We will often
write
\[
s(v,w) := s(v)(w).
\]
Under the identification above the set $\widetilde{X}(V)$ of inner products on
$V$ becomes the open subset of $\Sym(V)$ consisting of those elements $s \in
\Sym(V)$ for which $s \colon V \to V^*$ is bijective and $s(v,v) \ge 0$
holds for all $v \in V$, or, equivalently, for which $s(v,v) \ge 0$ holds for
all $v \in V$ and we have $s(v,v) = 0 \Leftrightarrow v = 0$.  In particular
$\widetilde{X}(V)$ inherits from the vector space $\Sym(V)$ the structure of a
smooth manifold.

Given a linear map $f \colon V \to W$, we obtain a linear map $\Sym(f) \colon
\Sym(W) \to \Sym(V)$ by sending $s \colon W \to W^*$ to $f^* \circ s \circ f$. A
linear isomorphism $f \colon V \to W$ induces a bijection $\widetilde{X}(f)
\colon \widetilde{X}(W) \to \widetilde{X}(V)$.  Obviously this is a
contravariant functor, i.e., $\widetilde{X}(g \circ f) = \widetilde{X}(f) \circ
\widetilde{X}(g)$.  If $\aut(V)$ is the group of linear automorphisms of $V$, we
obtain a right $\aut(V)$-action on $\widetilde{X}(V)$.

If $f \colon V \to W$ is a linear map and $s_V$ and $s_W$ are inner products on
$V$ and $W$, then the adjoint of $f$ with respect to these inner products is
$s_V^{-1} \circ f^* \circ s_W \colon W \to V$.

Consider a natural
number $m \le n:=\dim(V)$.  There is a canonical isomorphism
\begin{eqnarray*}
  \beta_m(V) \colon \Lambda^mV^* \xrightarrow{\cong} (\Lambda^mV)^*
  %\label{identification_beta_M_Lambda_dual}
\end{eqnarray*}
which maps $\alpha_1 \wedge \alpha_2 \wedge \ldots \wedge \alpha_m$ to the map
$\Lambda^m V \to \IR$ sending $v_1 \wedge v_2 \wedge \cdots \wedge v_m$ to
$\sum_{\sigma \in S_m} \sign(\sigma) \cdot \prod_{i = 1}^m
\alpha_i(v_{\sigma(i)})$.  
Let $s \colon V \to V^*$ be an inner product on $V$.
We obtain an inner product $s_{\Lambda^m}$ on
$\Lambda^m V$ by the composite
\[
s_{\Lambda^m} \colon \Lambda ^m V \xrightarrow{\Lambda^m s} \Lambda ^m V^*
\xrightarrow{\beta_m(V)} (\Lambda^m V)^*.
\]
One easily checks by a direct calculation for elements $v_1,v_2, \ldots,
v_m,w_1,w_2, \ldots, w_m$ in $V$
\begin{eqnarray}
  s_{\Lambda^m}\bigl(v_1 \wedge \cdots \wedge v_m,w_1 \wedge \cdots \wedge w_m\bigr)
  & = & 
  \det\bigl((s(v_i,w_j))_{i,j}\bigr),
  \label{Lambdams_and_det}
\end{eqnarray}
where $(s(v_i,w_j))_{i,j}$ is the obvious symmetric $(m,m)$-matrix.

Next we want to define a Riemannian metric $g$ on $\widetilde{X}(V)$. Since
$\widetilde{X}(V)$ is an open subset of $\Sym(V)$ and we have a canonical
trivialization $\phi_{\Sym(V)}$ of $T\Sym(V)$
(see~\eqref{trivialization_phi_V_of_TV}), we have to define for every $s \in
\widetilde{X}(V)$ an inner product $g_s$ on $\Sym(V)$. It is given by
\[
g_s(u,v) := \tr(s^{-1} \circ v \circ s^{-1} \circ u),
\]
for $u,v \in \Sym(V)$.
Here $\tr$ denotes the trace of endomorphisms of $V$.
Obviously $g_s(-,-)$ is bilinear and
symmetric since the trace is linear and satisfies $\tr(ab) = \tr(ba)$.  Since
$s^{-1} \circ (s^{-1} \circ u)^* \circ s = s^{-1}\circ u$ holds, the
endomorphism $s^{-1} \circ u \colon V \to V$ is selfadjoint with respect to the
inner product $s$.  Hence $\tr(s^{-1} \circ u \circ s^{-1} \circ u) \ge 0$
holds and we have $\tr(s^{-1} \circ u \circ s^{-1} \circ u) = 0$ if and only if
$s^{-1} \circ u= 0$, or, equivalently, $u = 0$. Hence $g_s$ is an inner product
on $\Sym(V)$.  We omit the proof that $g_s$ depends smoothly on $s$. Thus we
obtain a Riemannian metric $g$ on $\widetilde{X}(V)$.

\begin{lemma} \label{lem:metric__aut(V)_invariant} The Riemannian metric on
  $\widetilde{X}$ is $\aut(V)$-invariant.
\end{lemma}
\begin{proof}
  We have to show for an automorphism $f \colon V \to V$, an element $s \in
  \widetilde{X}(V)$ and two elements $u,v \in T_s\widetilde{X}(V) = T_s\Sym(V) =
  \Sym(V)$ that
  \[
  g_{\widetilde{X}(f)(s)}\bigl(T_s\widetilde{X}(f)(u),
  T_s\widetilde{X}(f)(v)\bigr) = g_s(u,v)
  \]
  holds. This follows from the following calculation:
  \begin{eqnarray*}
    \lefteqn{g_{\widetilde{X}(f)(s)}\bigl(T_s\widetilde{X}(f)(u), T_s\widetilde{X}(f)(v)\bigr)}
    &  &
    \\
    & = &  
    \tr\left(\widetilde{X}(f)(s)^{-1} \circ T_s\widetilde{X}(f)(v) \circ \widetilde{X}(f)(s)^{-1} \circ T_s\widetilde{X}(f)(u)\right)
    \\
    & = &
    \tr\left((f^* \circ s \circ f)^{-1} \circ (f^* \circ v \circ f) \circ (f^* \circ s \circ f)^{-1}\circ (f^* \circ u \circ f) \right)
    \\
    & = &
    \tr\left(f^{-1}  \circ s^{-1} \circ (f^*)^{-1} \circ f^* \circ v \circ f \circ f^{-1}  \circ s^{-1} \circ (f^*)^{-1} \circ f^* \circ u \circ f \right)
    \\
    & = & 
    \tr\left(f^{-1}  \circ s^{-1} \circ  v \circ s^{-1} \circ u \circ f \right)
    \\
    & = & 
    \tr\left(s^{-1} \circ  v \circ s^{-1} \circ u \circ f \circ f^{-1} \right)
    \\
    & = & 
    \tr\left(s^{-1} \circ  v \circ s^{-1} \circ u\right)
    \\
    & = & 
    g_s(u,v).
  \end{eqnarray*}
\end{proof}

Recall that so far we worked with the natural right
action of $\aut(V)$ on $\widetilde{X}(V)$.
In the sequel
we prefer to work with the corresponding left action obtained by precomposing with
$f \mapsto f^{-1}$ in order to match with standard notation, compare
in particular
diagram \eqref{diagram-GL-SL} below. 

Fix a base point $s_0 \in \widetilde{X}(V)$. Choose a linear isomorphism $\IR^n
\xrightarrow{\cong} V$ which is isometric with respect to the standard inner
product on $\IR^n$ and $s_0$. It induces an isomorphism $\GL_n(\IR)
\xrightarrow{\cong} \aut(V)$ and thus a smooth left action $\rho \colon
\GL_n(\IR) \times \widetilde{X}(V) \to \widetilde{X}(V)$. 
Since for any two elements $s_1,
s_2 \in \widetilde{X}(V)$ there exists an automorphism of $V$ which is an
isometry $(V,s_1) \to (V,s_2)$, this action is transitive. The stabilizer of
$s_0 \in \widetilde{X}(V)$ is the compact subgroup $O(n) \subseteq \GL_n(\IR)$. Thus
we obtain a diffeomorphism $\widetilde{\phi} \colon
\GL_n(\IR) / O(n)
\xrightarrow{\cong} \widetilde{X}(V)$. Define smooth maps $p \colon
\widetilde{X}(V) \to \IR^{>0}, \; s \mapsto \det(s^{-1}_0 \circ s)$ and $q \colon
\GL_n(\IR) / O(n) \to \IR^{>0}, \; A \cdot O(n)  \mapsto \det(A)^2$. 
Both maps are
submersions.  In particular the preimages of $1 \in \IR^{>0}$ under $p$ and $q$
are submanifolds of codimension $1$.  Denote by $\widetilde{X}(V)_1 :=
p^{-1}(1)$ and let $i \colon \widetilde{X}(V)_1 \to \widetilde{X}(V)$ be the
inclusion.  
Set $\SL_n^{\pm}( \IR ) = \{ A
\in \GL_n ( \IR) \; | \; \det (A) =  \pm 1 \}$.
The smooth left action $\rho \colon \GL_n(\IR) \times \widetilde{X}(V) \to
\widetilde{X}(V)$ restricts to an action  $\rho_1 \colon \SL_n^{\pm}(\IR)
\times \widetilde{X}(V)_1 \to \widetilde{X}(V)_1$. Since $f^* \circ
s_0 \circ f \in \widetilde{X}(V)_1$ implies $1 = \det (s_0^{-1} \circ
f^* \circ s_0 \circ f ) =  \det(f)^2$ this action is still transitive. The stabilizer of $s_0$ is
still $O(n) \subseteq \SL_n^{\pm}(\IR)$ and thus we obtain a
diffeomorphism $\widetilde{\phi}_1 \colon \SL_n^{\pm}(\IR)/O(n) \to
\widetilde{X}(V)_1$. Note that the inclusion induces a diffeomorphism
$\SL_n ( \IR ) / SO (n) \cong \SL_n^{\pm}(\IR) / O (n)$ but we prefer
the right hand side because we are
interested in the $\GL_n ( \IZ )$-action. The inclusion $\SL_n^{\pm}(\IR) \to \GL_n(\IR)$ induces an
embedding $j \colon \SL_n^{\pm}(\IR)/O(n) \to \GL_n(\IR)/O(n)$ with image $q^{-1}(1)$.
One easily checks that the following diagram commutes
\begin{equation} \label{diagram-GL-SL}
\xymatrix{\widetilde{X}(V)_1 \ar[r]^{i} & \widetilde{X}(V) \ar[r]^{p} & \IR^{>0}
  \\
  \SL_n^{\pm}(\IR)/O(n) \ar[r]^j \ar[u]^{\widetilde{\phi}_1} & \GL_n(\IR)/O(n) \ar[r]^(0.65){q}
  \ar[u]^{\widetilde{\phi}} & \IR^{>0} \ar[u]^{\id} }
\end{equation}
Equip $\widetilde{X}(V)_1$ with the Riemannian metric $g_1$ obtained from the
Riemannian metric $g$ on $\widetilde{X}(V)$.  We conclude from
Lemma~\ref{lem:metric__aut(V)_invariant} that $g_1$ is $\SL_n^{\pm}(\IR)$-invariant.
Since $\SL_n^{\pm}(\IR)$ is a semisimple Lie group with finite center and $O(n)
\subseteq \SL_n^{\pm}(\IR)$ is a maximal compact subgroup, $\widetilde{X}(V)_1$ is a
symmetric space of non-compact type and its sectional curvature is non-positive
(see~\cite[Section~2.2 on p.~70]{Eberlein(1996)},\cite[Theorem~3.1~(ii) in~V.3 on p.~241]{Helgason(1978)}).
Alternatively~\cite[Chapter~II, Theorem~10.39 on p.~318 and
Lemma~10.52 on p.~324]{Bridson-Haefliger(1999)} show that
$\widetilde{X}(V)$ and $\widetilde{X}(V)_1$ are proper CAT(0) spaces.

We call two elements $s_1,s_2 \in \widetilde{X}(V)$ equivalent if there exists $r \in \IR^{>0}$
with $r \cdot s_1 = s_2$. Denote by $X(V)$ the set of equivalence classes under this equivalence relation.
Let $\pr \colon \widetilde{X}(V) \to X(V)$ be the projection and equip $X(V)$ with the quotient topology.
The composite $\pr \circ i \colon \widetilde{X}(V)_1 \to X(V)$ is a homeomorphism. In the sequel we equip
$X(V)$ with the structure of a Riemannian manifold for which $\pr
\circ i$ is an isometric diffeomorphism.
In particular $X(V)$ is a $\CAT(0)$-space. The $\GL_n (\IR)$-action on
$\widetilde{X}(V)$ descends to an action on $X(V)$ and
the diffeomorphism $\pr \circ i$ is $\SL_n^{\pm}(\IR)$-equivariant.

%%%%%%%%%%%%%%%%%%%%%%%%%%%%%%%%%%%%%%%%%%%%%%%%%%%%%%%%%%%%%%%%%%%%%%%%%

\subsection{The volume function}
\label{subsec:The_volume_function}

Fix an integer $m$ with $1 \le m \le n = \dim(V)$. In this section we investigate the following volume function.

\begin{definition}[Volume function]
  Consider $\xi \in \Lambda^mV$ with $\xi \not= 0$. Define the \emph{volume function
    associated to $\xi$} by
  \[
  \vol_\xi \colon \widetilde{X}(V) \to \IR, \quad s \mapsto
  \sqrt{(s_{\Lambda^m})(\xi,\xi)},
  \]
  i.e., the function $\vol_{\xi}$ sends an inner product $s$ on $V$ to
  the length of $\xi$ with respect to $s_{\Lambda^m}$.
\end{definition}

Fix $\xi \neq 0$ in $\Lambda^m V$ of the form $\xi =v_1 \wedge v_2 \cdots \wedge v_m$ for the
rest of this section. This means that the line $\langle \xi\rangle$ lies in the image of
the Pl\"ucker embedding, 
compare \cite[Chapter~1.5 p.~209--211]{Griffith-Harris(1978)}.  
Then there is precisely one $m$-dimensional subvector
space $V_{\xi} \subseteq V$ such that the image of the map 
$\Lambda^m V_{\xi} \to \Lambda^mV$
induced by the
inclusion is the $1$-dimensional subvector
space spanned by $\xi$.  The subspace $V_{\xi}$ is the span of the vectors $v_1, v_2, \ldots, v_m$. 
It can be expressed as $V_\xi = \{v\in V|v\wedge \xi=0\}$. This shows that
$V_\xi$ depends on the line $\langle \xi \rangle$ but is independent of the choice of $v_1,\ldots,v_m$.

Given an inner product $s$ on $V$, we obtain an
orthogonal decomposition $V = V_{\xi} \oplus V_{\xi}^{\perp}$ of $V$ with
respect to $s$ and we define $s_{\xi} \in \Sym(V)$ to be the
element which satisfies $s_{\xi}(v + v^{\perp},w + w^{\perp}) = s(v,w)$ for all
$v,w \in V_{\xi}$ and $v^{\perp},w^{\perp} \in V^{\perp}_{\xi}$.

\begin{theorem}[Gradient of the volume function]
\label{the:gradient_of_the_volume_function}
The gradient of the square $\vol_{\xi}^2 $ of the volume function
$\vol_{\xi}$ is given for $s \in \widetilde{X}(V)$ by
\[
\nabla_s(\vol_{\xi}^2) = \vol_{\xi}^2(s)  \cdot s_{\xi} \in \Sym(V) = T_s\widetilde{X}(V).
\]
\end{theorem}
\begin{proof}
Let $A$ be any $(m,m)$-matrix. Then 
\begin{eqnarray}
\lim_{t \to 0} \frac{\det(I_m + t \cdot A) - \det(I_m)}{t}
& = & \tr(A).
\label{derivative_of_det_is_trace}
\end{eqnarray}
This follows because 
\[
\det (I_m + t \cdot A) = 1 + t \tr(A) \; \mbox{ mod } \; t^2
\]
by the Leibniz formula for the determinant.

Consider $s \in \widetilde{X}(V)$ and $u \in T_s(\widetilde{X}(V)) = \Sym(V)$. 
Notice that there exists  $\epsilon > 0$ such that $s + t \cdot u$ 
lies in $\widetilde{X}(V)$ for all $t \in (-\epsilon,\epsilon)$.  Fix $v_1, v_2, \ldots , v_m \in V$ 
with $\xi = v_1 \wedge \cdots \wedge v_m$.  We compute
using~\eqref{Lambdams_and_det} and~\eqref{derivative_of_det_is_trace}
\begin{eqnarray*}
\lefteqn{\lim_{t \to 0} \frac{\vol_{\xi}^2(s + t \cdot u) - \vol_{\xi}^2(s)}{t}}
& & 
\\
& = & 
\lim_{t \to 0} \frac{\det\bigl(((s+ t \cdot u)(v_i,v_j))_{i,j}\bigr) - \det\bigl((s(v_i,v_j))_{i,j}\bigr)}{t}
\\
& = & 
\lim_{t \to 0} \frac{\det\bigl((s(v_i,v_j))_{i,j} + t \cdot (u(v_i,v_j))_{i,j}\bigr) - \det\bigl((s(v_i,v_j))_{i,j}\bigr)}{t}
\\
& = & 
\det\bigl((s(v_i,v_j))_{i,j}\bigr) \cdot \lim_{t \to 0} 
\frac{\det \bigl(I_m + t \cdot (s(v_i,v_j))_{i,j}^{-1} \cdot (u(v_i,v_j))_{i,j}\bigr) - \det\bigl(I_m\bigr)}{t}
\\
& = & 
\det\bigl((s(v_i,v_j))_{i,j}\bigr) \cdot \tr\bigl((s(v_i,v_j))_{i,j}^{-1} \cdot (u(v_i,v_j))_{i,j}\bigr).
\end{eqnarray*}
The gradient $\nabla_s(\vol_{\xi}^2)$ is uniquely determined by the property
that for all $u \in \Sym(V)$ we have
\[
g_s\bigl(\nabla_s(\vol_{\xi}^2),u\bigr) 
= \lim_{t \to 0} \frac{\vol_{\xi}^2(s + t \cdot u) - \vol_{\xi}^2(s)}{t}.
\]
Hence it remains to show for every $u \in \Sym(V)$
\[
g_s\bigl(\vol_{\xi}^2(s)  \cdot s_{\xi},u\bigr) 
=  \det\bigl((s(v_i,v_j))_{i,j}\bigr) \cdot \tr\bigl((s(v_i,v_j))_{i,j}^{-1} \cdot (u(v_i,v_j))_{i,j}\bigr).
\]
Since $\vol_{\xi}^2(s) = \det\bigl((s(v_i,v_j))_{i,j}\bigr)$
by~\eqref{Lambdams_and_det} and $g_s(s_{\xi},u) = \tr(s^{-1} \circ s_{\xi} \circ s^{-1} \circ u)$ by definition,
it remains to show
\[
\tr(s^{-1} \circ s_{\xi} \circ s^{-1} \circ u) =  \tr\bigl((s(v_i,v_j))_{i,j}^{-1} \cdot (u(v_i,v_j))_{i,j}\bigr).
\]
We obtain a decomposition 
$s = \left( \begin{matrix} s_{\xi}  & 0 \\ 0 & 
    s_{\xi}^{\perp} \end{matrix} \right)$
for an inner
product $s_{\xi}^{\perp} \colon V_{\xi}^{\perp} \to (V_{\xi}^{\perp})^*$, where we
will here and in the sequel identify $(V_{\xi})^* \oplus (V_{\xi}^{\perp})^*$
and $(V_{\xi} \oplus V_{\xi}^{\perp})^*$ by the canonical isomorphism.  We
decompose
\[
u = \left(\begin{matrix} u_{\xi} & u' \\ u'' & u''' \end{matrix}\right) \colon V_{\xi} \oplus V_{\xi}^{\perp} \to
(V_{\xi})^* \oplus (V_{\xi}^{\perp})^*
\]
for a linear map $u_{\xi} \colon V_{\xi} \to V_{\xi}^*$. One easily checks
\[
\tr(s^{-1} \circ s_{\xi} \circ s^{-1} \circ u) = \tr(s_\xi^{-1}\circ u_\xi).
\] 
The set $\{v_1, v_2, \ldots, v_m\}$ is a basis for $V_{\xi}$. Let 
$\{v_1^*, v_2^*, \ldots, v_m^*\}$ be the dual basis of $V_{\xi}^*$. Then the matrix of
$u_{\xi}$ with respect to these basis is $(u(v_i,v_j))_{i,j}$ and the matrix of
$s_{\xi}$ with respect to these basis is $(s(v_i,v_j))_{i,j}$. Hence the matrix
of $s_{\xi}^{-1} \circ u_{\xi} \colon V_{\xi} \to V_{\xi}$ with respect to the
basis $\{v_1, v_2, \ldots v_m\}$ is $(s(v_i,v_j))_{i,j}^{-1} \cdot
(u(v_i,v_j))_{i,j}$. This implies
\[
\tr(s_{\xi}^{-1} \circ u_{\xi}) = \tr\bigl((s(v_i,v_j))_{i,j}^{-1} \cdot (u(v_i,v_j))_{i,j}\bigr).
\]
This finishes the proof of Theorem~\ref{the:gradient_of_the_volume_function}.
\end{proof}

\begin{corollary}\label{cor:gradient_of_log(vol2)}
  The gradient of the function
  \[\ln \circ \vol_{\xi}^2 \colon \widetilde{X}(V) \to \IR\]
  at $s \in \widetilde{X}(V)$ is given by the tangent vector $s_{\xi} \in
  T_s\widetilde{X}(V) = \Sym(V)$. In particular its norm with respect to the
  Riemannian metric $g$ on $\widetilde{X}(V)$ is independent of $s$, namely,
  $\sqrt{m}$.
\end{corollary}
\begin{proof}
  Since the derivative of $\ln(x)$ is $x^{-1}$, the chain rule implies together
  with Theorem~\ref{the:gradient_of_the_volume_function} for $s \in
  \widetilde{X}(V)$
  \begin{equation*}
    \nabla_s\bigl(\ln \circ \vol_{\xi}^2\bigr) 
    =  
    \nabla_s\bigl(\vol(\xi)^2\bigr)\cdot \frac{1}{\vol_{\xi}^2(s)}     
     =  
    \vol_{\xi}^2(s)  \cdot s_{\xi} \cdot \frac{1}{\vol_{\xi}^2(s)} 
     =  
    s_{\xi}.
  \end{equation*}
  We use the orthogonal decomposition $V=V_\xi\oplus V_\xi^\perp$ with respect to $s$ to compute 
  \[ 
  \begin{split}
    \left(  ||\nabla_s(\ln \circ  \vol_{\xi}^2)||_s\right)^2 &
      \; =  \; 
    \left(||s_{\xi}||_s\right)^2    
     \; =  \; 
    g_s\bigl(s_{\xi},s_{\xi}\bigr) 
      \; =  \; 
    \tr\bigl(s^{-1} \circ s_{\xi} \circ s^{-1} \circ s_{\xi}\bigr)
    \\ &
       \; =  \;   
    \tr\bigl(\id_{V_{\xi}} \oplus \,0 \colon V_{\xi} \oplus V_{\xi}^{\perp} 
               \to V_{\xi} \oplus V_{\xi}^{\perp}\bigr)
      \; =  \;  
    \dim(V_{\xi})  \; =  \;  m.
  \end{split} 
  \]
\end{proof}

%%%%%%%%%%%%%%%%%%%%%%%%%%%%%%%%%%%%%%%%%%%%%%%%%%%%%%%%%%%%%%%%%%%%%%%% 
%%%%%%%%%%% Section: Sublattices of Z^n  %%%%%%%%%%%%%%%%%%%%%%%%%%%%%%%
%%%%%%%%%%%%%%%%%%%%%%%%%%%%%%%%%%%%%%%%%%%%%%%%%%%%%%%%%%%%%%%%%%%%%%%%

\typeout{------------------------  Section 2:  Sublattices of Z^n --------------}

\section{Sublattices of \texorpdfstring{$\IZ^n$}{Zn}}
\label{sec:Sublattices_of_Zn}

A \emph{sublattice} $W$ of $\IZ^n$ is a $\IZ$-submodule $W \subseteq \IZ^n$ such
that $\IZ^n/W$ is a projective $\IZ$-module. Equivalently, $W$ is a
$\IZ$-submodule $W \subseteq \IZ^n$ such that for $x \in \IZ^n$ for which $k
\cdot x$ belongs to $W$ for some $k \in \IZ, k \not= 0$ we have $x \in W$.  Let
$\call$ be the set of sublattices $L$ of $\IZ^n$.  

Consider $W \in \call$. Let $m$ be its rank as an abelian group. Let
$\Lambda_{\IZ}^m W \to \Lambda_{\IZ}^m\IZ^n \to \Lambda^m\IR^n$ be the obvious
map. Let $\xi(W) \in \Lambda^m\IR^n$ be the image of a generator of the infinite
cyclic group $\Lambda_{\IZ}^m W$.  We have defined the map $\vol_{\xi(W)} \colon
\widetilde{X}(\IR^n) \to \IR$ above. Obviously it does not change if we replace
$\xi(W)$ by $- \xi(W)$.  Hence it depends only on $W$ and not on the choice of
generator $\xi(W)$. Notice that for $W = 0$ we have by definition
$\Lambda^0_{\IZ} W = \IZ$ and $\Lambda^0 \IR^n = \IR$ and $\xi(W)$ is $\pm
1 \in \IR$. In that case $\vol_{\xi}$ is the constant
function with value $1$.

We will abbreviate
\begin{equation*}
  \widetilde{X} =  \widetilde{X}(\IR^n), \;
  \widetilde{X}_1 =  \widetilde{X}(\IR^n)_1, \;
  X  =  X(\IR^n) \; \text{and} \;
  \vol_W  =  \vol_{\xi(W)} \; \text{for}\; W \in \call.
\end{equation*}
Given a chain   $W_0 \subsetneq W_1$ of elements $W_0, W_1 \in \call$,  we define a function
\begin{eqnarray*}
\widetilde{c}_{W_0 \subsetneq W_1} \colon \widetilde{X} \to \IR
\end{eqnarray*}
by
\[\widetilde{c}_{W_0 \subsetneq W_1}(s) 
:= \frac{\ln\bigl(\vol_{W_1}(s)\bigr) - \ln\bigl(\vol_{W_0}(s)\bigr)}{\rk_{\IZ}(W_1) - \rk_{\IZ}(W_0)}.
\]
Obviously this can be rewritten as
\begin{eqnarray*}
\widetilde{c}_{W_0 \subsetneq W_1}(s) 
= \frac{1}{2} \cdot \left(\frac{\ln\bigl(\vol_{W_1}(s)^2\bigr) 
- \ln\bigl(\vol_{W_0}(s)^2\bigr)}{\rk_{\IZ}(W_1) - \rk_{\IZ}(W_0)} \right).
\end{eqnarray*}
Hence we get for the norm of  the gradient of this function at $s \in \widetilde{X}$
\begin{eqnarray*}
\lefteqn{\left|\left|\nabla_s \bigl(\widetilde{c}_{W_0\subsetneq W_1}\bigr)\right|\right|}
& & 
\\
& = & 
\left|\left|\frac{1}{2} \cdot \left(\frac{\nabla_s\bigl(\ln \circ \vol_{W_1}^2\bigr) 
- \nabla_s\bigl(\ln \circ \vol_{W_0}^2\bigr)}{\rk_{\IZ}(W_1) - \rk_{\IZ}(W_0)} \right)\right|\right|
\\
& \le &
\frac{1}{2} \cdot \left(\frac{\left|\left|\nabla_s\bigl(\ln \circ \vol_{W_1}^2\bigr)\right|\right| 
+ \left|\left|\nabla_s\bigl(\ln \circ \vol_{W_0}^2\bigr)\right|\right| }{\rk_{\IZ}(W_1) - \rk_{\IZ}(W_0)}\right)
\\
& \le & 
\frac{1}{2} \cdot \left( \left|\left|\nabla_s\bigl(\ln \circ \vol_{W_1}^2\bigr)\right|\right| 
+ \left|\left|\nabla_s\bigl(\ln \circ \vol_{W_0}^2\bigr)\right|\right| \right).
\end{eqnarray*}
We conclude from Corollary~\ref{cor:gradient_of_log(vol2)} for all $s \in \widetilde{X}$.
\[
\left|\left|\nabla_s \bigl(\widetilde{c}_{W_0\subsetneq W_1}\bigr)\right|\right|
\le
\frac{\sqrt{\rk_{\IZ}(W_1)}  + \sqrt{\rk_{\IZ}(W_0)}}{2}
\le 
\sqrt{n}.
\]
If $f \colon M \to \IR$ is a differentiable function on a Riemannian
manifold $M$ and $C = \sup \{ || \nabla_x f || \; | \; x \in M \}$ we always have
\[
|f(x_1) - f(x_2) | \leq C d_M(x_1 , x_2),
\]
where $d_M$ denotes the metric associated to the Riemannian metric.
In particular we
get for any two elements $s_0,s_1 \in \widetilde{X}$
\begin{eqnarray*}
\left|\widetilde{c}_{W_0 \subsetneq  W_1}(s_1) - \widetilde{c}_{W_0 \subsetneq  W_1}(s_0)\right|
& \le &
\sqrt{n} \cdot d_{\widetilde{X}}(s_0,s_1),
\end{eqnarray*}
where $ d_{\widetilde{X}}$ is the metric on $\widetilde{X}$ coming from the Riemannian
metric $g$ on $\widetilde{X}$.  Recall that the Riemannian metric $g_1$
on $\widetilde{X}_1$ is obtained by restricting the metric $g$.  Let
$d_{\widetilde{X}_1}$ be the metric on $\widetilde{X}_1$ coming
from the Riemannian metric $g_1$ on $\widetilde{X}_1$.  Recall that 
$i\colon \widetilde{X}_1 \to \widetilde{X}$ is the inclusion.  Then
we get for $s_0,s_1 \in \widetilde{X}_1$
\begin{eqnarray*}
d_{\widetilde{X}}(i(s_0),i(s_1)) \le d_{\widetilde{X}_1}(s_0,s_1).
\end{eqnarray*}
Indeed $\widetilde{X}_1$ is a geodesic submanifold of $\tilde{X}$
(\cite[Chapter~II, Lemma~10.52 on p.~324]{Bridson-Haefliger(1999)}). 
So both sides are even equal. We obtain for
$s_0,s_1 \in \widetilde{X}_1$
\begin{eqnarray}
\left|\widetilde{c}_{W_0 \subsetneq  W_1} \circ i(s_1) - \widetilde{c}_{W_0 \subsetneq  W_1} \circ i(s_0)\right|
& \le &
\sqrt{n} \cdot d_{\widetilde{X}_1}(s_0,s_1).
\label{inequality_for_tildec_W_0-W_1}
\end{eqnarray}
Put
\begin{eqnarray*}
\call' & = & \{W \in \call \mid W \not= 0, W \not=\IZ^n\}.
\end{eqnarray*}
Define for $W \in \call'$  functions
\[
\widetilde{c}_W^i, \widetilde{c}_W^s \colon \widetilde{X} \to \IR
\]
by
\begin{eqnarray*}
\widetilde{c}_W^i(x) &:= & \inf\bigl\{\widetilde{c}_{W \subsetneq W_2}(x) \mid W_2 \in  \call, W \subsetneq W_2\bigr\};
\\
\widetilde{c}_W^s(x) &:= & \sup\bigl\{\widetilde{c}_{W_0 \subsetneq W}(x) \mid W_0 \in  \call, W_0 \subsetneq W\bigr\}.
\end{eqnarray*}
In order to see that infimum and supremum exist we use the fact that
for fixed $x \in \widetilde{X}$ there are at most finitely many $W \in
\call$ with $\vol_W (x) \leq 1$, compare \cite[Lemma~1.15]{Grayson(1984)}.
Put
\begin{eqnarray}
\widetilde{d}_W \colon \widetilde{X} \to \IR, & & x \mapsto \exp \bigl(\widetilde{c}_W^i(x) - \widetilde{c}_W^s(x)\bigr).
\label{d_W}
\end{eqnarray}
Since $\vol_{W}(r \cdot s) = r^{\rk_{\IZ}(W)} \cdot \vol_{W }(s)$
 holds for $r \in \IR^{>0}$, $W \in \call$ and $s \in \widetilde{X}$,
 we have $\widetilde{c}_{W_0 \subsetneq W_1} (r \cdot s ) =
 \widetilde{c}_{W_0 \subsetneq W_1} (s) + \ln (r)$ and  the
function $\widetilde{d}_W$ factorizes over the
projection $\pr \colon \widetilde{X} \to X$ to a function
\begin{eqnarray*}
d_W \colon X  \to  \IR.
\end{eqnarray*}

\begin{lemma}
  \label{lem:estimate_for_d_W}
  Consider $x \in X$, $W \in \call'$, and $\alpha\in \IR^{>0}$. Then we
  get for all $y \in X$ with $d_{X}(x,y) \le \alpha$
  \[
  d_W(y) \in [d_W(x)/e^{2\sqrt{n}\alpha},d_W(x) \cdot e^{2\sqrt{n}\alpha}].
  \]
\end{lemma}
\begin{proof}
By the definition of the structure of a smooth Riemannian manifold on $X$,
it suffices to show for all $s_0,s_1 \in \widetilde{X}_1$ with
$d_{\widetilde{X}_1}(s_0 ,s_1 ) \le \alpha$
\[
 \widetilde{d}_W(s_1) \in [\widetilde{d}_W(s_0)/e^{2\sqrt{n}\alpha},\widetilde{d}_W(s_0) \cdot e^{2\sqrt{n}\alpha}].
  \]
One concludes from~\eqref{inequality_for_tildec_W_0-W_1} for $s_0,s_1 \in \widetilde{X}_1$ with $d_{\widetilde{X}_1}(s_0,s_1) \le \alpha$
and $W \in\call'$
\begin{eqnarray*}
\widetilde{c}_W^s(s_1) & \in & [\widetilde{c}_W^s(s_0)- \sqrt{n} \cdot \alpha,\widetilde{c}_W^s(s_0) + \sqrt{n} \cdot \alpha],
\\
\widetilde{c}_W^i(s_1) & \in & [\widetilde{c}_W^i(s_0)- \sqrt{n} \cdot \alpha,\widetilde{c}_W^i(s_0) + \sqrt{n} \cdot \alpha].
\end{eqnarray*}
Now the claim follows.
\end{proof}
In particular the function $d_W \colon X \to \IR$ is continuous.

Define the following open subset of $X$ for $W \in \call'$ and $t \ge 1$
\begin{eqnarray*}
  X(W,t) & := & \{x \in X \mid d_W(x) > t\}.
  %\label{X(W,t)}
\end{eqnarray*}
There is an obvious $\GL_n(\IZ)$-action on $\call$ and $\call'$ and
in the discussion following Lemma~\ref{lem:metric__aut(V)_invariant} we have
already explained how $\GL_n(\IZ) \subset \SL^{\pm}_n (\IR)$ acts on $\widetilde{X}_1$ and hence on $X$ (choosing the
standard inner product on $\IR^n$ as the basepoint $s_0$).

\begin{lemma} For any $t\ge 1$ we get:
  \label{lem:Grayson}\
  \begin{enumerate}

  \item \label{lem:Grayson:G-invariance} $X(gW,t) = gX(W,t)$ for $g \in
    \GL_n(\IZ)$, $ W \in \call'$;
  \item \label{lem:Grayson:G-compact} The complement of the
    $\GL_n(\IZ)$-invariant open subset 
    \[
     |\calw(t)| := \bigcup_{W\in \call'} X(W,t)
    \]
    in
    $X$ is a cocompact $\GL_n(\IZ)$-set;

  \item \label{lem:Grayson:intersection} If $X(W_1,t) \cap X(W_2,t)
    \cap \cdots \cap 
    X(W_k,t) \not= \emptyset$ for $W_i \in \call'$, then we can find a permutation $\sigma \in
    \Sigma_k$ such that $W_{\sigma(1)} \subseteq W_{\sigma(2)} \subseteq \ldots
    \subseteq W_{\sigma(k)}$ holds.

  \end{enumerate}
\end{lemma}
\begin{proof}
  This follows directly from Grayson~\cite[Lemma~2.1, Cor.~5.2]{Grayson(1984)} 
  as soon as
  we have explained how our setup corresponds to the one of Grayson.

We are only dealing with the case $\calo = \IZ$ and $F = \IQ$
of~\cite{Grayson(1984)}.  In particular there is only one archimedian place,
namely, the absolute value on $\IQ$ and for it $\IQ_{\infty} = \IR$. So an
element $s \in \widetilde{X}$ corresponds to the structure of a lattice
which we will denote by $(\IZ^n,s)$ with
underlying $\IZ$-module $\IZ^n$ in the sense of~\cite{Grayson(1984)}.  Given $s
\in \widetilde{X}$,  an element $W \in \call$ defines a sublattice  of the lattice $(\IZ^n,s)$ in the sense
of~\cite{Grayson(1984)} which we will denote by $(\IZ^n,s) \cap W$. The volume of a
sublattice $(\IZ^n,s) \cap W$ of $(\IZ^n,s)$ in the sense of~\cite{Grayson(1984)} is
$\vol_{W}(s)$.

Given $W \in \call'$, we obtain a $\IQ$-subspace in the sense
of~\cite[Definition~2.1]{Grayson(1984)} which we denote again by $W$, and vice
versa. It remains to explain why our function $d_W$ of~\eqref{d_W} agrees with
the function $d_W$ of~\cite[Definition~2.1]{Grayson(1984)} which is given by
\[
d_W(s) = \exp\biggl(\min\bigl((\IZ^n,s)/(\IZ^n,s) \cap W\bigr)  - \max \bigl((\IZ^n,s) \cap W\bigr)\biggr),
\]
see \cite[Definition~1.23, 1.9]{Grayson(1984)}.
This holds by the following observation.

Consider $s \in \widetilde{X}$ and $W \in \call'$. Consider the canonical plot and the
canonical polygon of the lattice $(\IZ^n,s) \cap W$ in the sense
of~\cite[Definition~1.10 and Discussion~1.16]{Grayson(1984)}. 
The slopes of the canonical polygon are strictly increasing when going from the left to the right because
of~\cite[Corollary~1.30]{Grayson(1984)}. Hence $\max \bigl((\IZ^n,s) \cap W\bigr)$ in the sense
of~\cite[Definition~1.23]{Grayson(1984)} is the slope of the segment of the
canonical polygon ending at $(\IZ^n,s) \cap W$.  Consider any $W_0 \in \call$ with $W_0 \subsetneq W$.
Obviously the slope of the line joining the plot point of $(\IZ^n,s) \cap W_0$
and $(\IZ^n,s) \cap W$ is less than or equal to the slope of the segment of the canonical polygon
ending at $(\IZ^n,s) \cap W$. If $(\IZ^n,s) \cap W_0$ happens to be the starting point of this segment, then
this slope agrees with the slope of the segment of the canonical plot ending at
$(\IZ^n,s) \cap W$. Hence
\begin{multline}
\max \bigl((\IZ^n,s) \cap W\bigr)
\\
=  \max \left\{\left.\frac{\ln(\vol_{W}(s)) - \ln(\vol_{W_0}(s))}{\rk_{\IZ}(W) - \rk_{\IZ}(W_0)}
\; \right|\; W_0 \in \call, W_0 \subsetneq W\right\} = \widetilde{c}_W^s(s).
\label{lem:Grayson:(s)}
\end{multline}
We have the formula 
\[
\vol\bigl((\IZ^n,s) \cap W_2\bigr) = \vol\bigl((\IZ^n,s) \cap W\bigr) \cdot \vol\bigl((\IZ^n,s) \cap W_2/(\IZ^n,s) \cap W\bigr)
\]
for any $W_2 \in \call$  with $W \subsetneq W_2$ (see~\cite[Lemma~1.8]{Grayson(1984)}). Hence
\begin{eqnarray*}
\frac{\ln(\vol_{W_2}(s)) - \ln(\vol_{W}(s))}{\rk_{\IZ}(W_2) - \rk_{\IZ}(W)}
& = & 
\frac{\ln(\vol_{W_2/W}(s))}{\rk_{\IZ}(W_2/W)}.
\end{eqnarray*} 
There is an obvious bijection of the set of direct summands
in $\IZ^n/W$ and the set of direct summand in $\IZ^n$ containing $W$. 
Now analogously to the proof of~\eqref{lem:Grayson:(s)} one shows
\begin{multline}
\min\bigl((\IZ^n,s)/((\IZ^n,s) \cap W)\bigr)
\\
 =  \min \left\{\left.\frac{\ln(\vol_{W_2}(s)) 
 -\ln(\vol_{W}(s))}{\rk_{\IZ}(W_2) - \rk_{\IZ}(W)}\; \right|\; W_2 \in \call, W \subsetneq W_2\right\} = \widetilde{c}_W^i(s).
\label{lem:Grayson:(i)}
\end{multline}
Now the equality of the two versions for $d_W$ follows
from~\eqref{lem:Grayson:(s)} and~\eqref{lem:Grayson:(i)}. This finishes the poof
of Lemma~\ref{lem:Grayson}.
\end{proof}

%===================================================================

\section{Transfer Reducibility of \texorpdfstring{$\GL_n(\IZ)$}{GLn(Z)}}
   \label{sec:Transfer_Reducibility_of_SL_n(IZ)}

  Let $\calf_n$ be the family of those subgroups $H$ of
  $\GL_n(\IZ)$, which are virtually cyclic or for which there 
  exists a finitely generated free abelian group $P$, natural
  numbers $r$ and $n_i$ with $n_i < n$ and an extension of
  groups 
  \[
  1 \to P \to K \to \prod_{i = 1} ^{r} \GL_{n_i}(\IZ) \to 1
  \] 
  such that $H$ is isomorphic to a subgroup of $K$. 

  In this section we prove the following theorem, which 
  by~\cite[Theorem~1.1]{Bartels-Lueck(2012annals)} implies the 
  $K$-theoretic FJC up to dimension $1$ for $\GL_n (\IZ)$ 
  with respect to the family $\calf_n$.
  The notion of \emph{transfer reducibility} has been 
  introduced in~\cite[Definition~1.8]{Bartels-Lueck(2012annals)}. 
  Transfer reducibility asserts the existence of a compact space $Z$ and certain
  equivariant covers of $G \x Z$.
  A slight
  modification of transfer reducibility is discussed in
  Section~\ref{sec:wreath-product_and_transfer-reducibility}, see Definition~\ref{def:almost_transfer_reducible}.

Here our main work is to verify the conditions formulated in
Definition~\ref{def:at-infinity_plus_periods} for our situation, see Lemma~\ref{lem:long_coverings}.
Once this has been done the verification of transfer reducibility
proceeds as in \cite{Bartels-Lueck(2010CAT(0)flow)}, as is explained after Lemma~\ref{lem:long_coverings}.

  \begin{theorem} \label{the:transfer-reducible}
    The group $\GL_n(\IZ)$ is transfer reducible over $\calf_n$.
  \end{theorem}

  To prove this we will use the space $X=X(\IR^n)$ and
  its subsets $X(W,t)$ considered in
  Sections~\ref{sec:The_space_of_inner_products_and_the_volume_function}
  and~\ref{sec:Sublattices_of_Zn}.

  For a $G$-space $X$ and a family of subgroups $\calf$, a subset 
  $U \subseteq X$ is called  an $\calf$-subset if
  $G_U := \{ g \in G \mid g(U) = U \}$ belongs to $\calf$ and
  $gU \cap U = \emptyset$ for all $g \in G \setminus G_U$.
  An open $G$-invariant cover consisting of $\calf$-subsets 
  is called an $\calf$-cover.

  \begin{lemma} \label{lem:covering_from_Grayson} 
    Consider for $t \ge   1$ the collection of subsets of $X$
    \[
    \calw(t) = \{X(W,t) \mid W \in \call'\}.
    \]
    It is a $\GL_n(\IZ)$-invariant set of open $\calf_n$-subsets of $X$ whose
    covering dimension is at most $(n-2)$.
  \end{lemma}
  
  \begin{proof}
    The set $\calw(t)$ is $\GL_n(\IZ)$-invariant because of
    Lemma~\ref{lem:Grayson}~\ref{lem:Grayson:G-invariance}
    and its covering dimension is
    bounded by $(n-2)$ because of 
    Lemma~\ref{lem:Grayson}~\ref{lem:Grayson:intersection} since for 
    any chain of sublattices 
    $\{0\} \subsetneq W_0 \subsetneq W_1 \subsetneq 
                 \cdots \subsetneq W_r \subsetneq \IZ^n$ of $\IZ^n$
    we have $r \le n-2$.

    It remains to show that $\calw (t)$ consists of $\calf_n$-subsets. 
    Consider $W \in \call'$ and $g \in \GL_n(\IZ)$.  
    Lemma~\ref{lem:Grayson} implies:
    \[
      X(W,t) \cap g X(W,t) \not= \emptyset \Longleftrightarrow 
      X(W,t)  \cap X(gW,t) \not= \emptyset \Longleftrightarrow gW = W.
    \]
    Put $\GL_n(\IZ)_W = \{g \in \GL_n(\IZ) \mid gW = W\}$. 
    Choose $V \subset \IZ^n$ with
    $W \oplus V = \IZ^n$.  Under this identification every element 
    $\phi \in \GL_n(\IZ)_W$ is of the shape
    \[
      \left(
       \begin{matrix}
        \phi_W & \phi_{V,W}
        \\
        0 & \phi_V
       \end{matrix}
      \right)
    \]
    for $\IZ$-automorphism $\phi_V$ of $V$, $\phi_W$ of $W$ and a 
    $\IZ$-homomorphism $\phi_{V,W} \colon V \to W$.  
    Define
    \[
     p \colon \GL_n(\IZ)_W \to \aut_{\IZ}(W) \times \aut_{\IZ}(V), \quad
     \left(\begin{matrix} \phi_W & \phi_{V,W} \\
                               0 & \phi_V
           \end{matrix}\right)
     \mapsto (\phi_W,\phi_V)
    \]
    and
    \[
     i \colon \hom_{\IZ}(V,W) \to \GL_n(\IZ)_W, \quad \psi \mapsto
     \left(\begin{matrix}  \id_W & \psi   \\
                               0 & \id_V
             \end{matrix}\right).
    \]
    Then we obtain an exact sequence of groups
    \[
     1 \to \hom_{\IZ}(V,W) \xrightarrow{i} \GL_n(\IZ)_W \xrightarrow{\pr}
       \aut_{\IZ}(W) \times \aut_{\IZ}(V) \to 1,
    \]
    where $\hom_{\IZ}(V,W)$ is the abelian group given by the obvious
    addition.     
    Since both $V$ and $W$ are different from $\IZ^n$,
    we get $\aut_{\IZ}(V) \cong \GL_{n(V)}(\IZ)$ and 
    $\aut_{\IZ}(W) \cong \GL_{n(W)}(\IZ)$  for $n(V), n(W) < n$. 
    Hence each element in $\calw(t)$ is an $\calf_n$-subset 
    with respect  to the $\GL_n(\IZ)$-action.
  \end{proof}

  For a subset $A$ of a metric space $X$ and $\alpha \in \IR^{>0}$ we denote by
  \[
    B_{\alpha}(A) := 
     \bigl\{x \in X \mid d_X(x,a) < \alpha \;\text{for some}\; a \in A\bigr\}.
  \]
  the $\alpha$-neighborhood of $A$.

  \begin{lemma} \label{lem:comparison} For every $\alpha > 0$ and $t \ge 1$ there
    exists $\beta > 0$ such that for every $W \in \call'$ we have
    $B_{\alpha}(X(W,t+ \beta)\bigr) \subseteq X(W,t)$.
  \end{lemma}

  \begin{proof}
    This follows from Lemma~\ref{lem:estimate_for_d_W} if we choose 
    $\beta > (e^{2\sqrt{n}\alpha} -1 ) \cdot t$.
  \end{proof}

  Let $\FS(X)$ be the flow space associated to the $\CAT(0)$-space 
  $X = X(\IR^n)$ in~\cite[Section~2]{Bartels-Lueck(2010CAT(0)flow)}.
  It consists of generalized geodesics, i.e., continuous maps 
  $c \colon \IR \to X$ for which there is a closed subinterval
  $I$ of $\IR$ such that $c|_{I}$ is a geodesic and $c|_{\IR \setminus I}$
  is locally constant.
  The flow $\Phi$ on $\FS(X)$ is defined by the formula 
  $(\Phi_\tau(c))(t) = c (\tau + t)$.

  \begin{lemma} \label{lem:estimate_for_d(0)}
    Consider $\delta, \tau > 0$ and $c \in \FS(X)$. 
    Then we get for $d \in B_{\delta}\bigl(\Phi_{[-\tau,\tau]}(c)\bigr)$
    \[
      d_X\bigl(d(0),c(0)\bigr) < 4 + \delta + \tau.
    \]
  \end{lemma}
  
  \begin{proof}
    Choose $s \in [-\tau,\tau]$ with $d_{\FS(X)}\bigl(d,\Phi_s(c)\bigr) < \delta$.
    We estimate using~\cite[Lemma~1.4~(i)]{Bartels-Lueck(2010CAT(0)flow)} 
    and a special case of~\cite[Lemma~1.3]{Bartels-Lueck(2010CAT(0)flow)}
    \begin{eqnarray*}
      d_X\bigl(d(0),c(0)\bigr) 
      & \le &
      d_X\bigl(d(0),\Phi_s(c)(0)\bigr) + d_X\bigl(\Phi_s(c)(0),c(0)\bigr)
      \\
      & \le & 
      d_{\FS(X)}\bigl(d,\Phi_s(c)\bigr) + 2  + d_{\FS(X)}\bigl(\Phi_s(c),c\bigr) +2
      \\
      & <  &
      \delta + 2 + |s| + 2
      \; \le \;  
      4 + \delta + \tau. 
    \end{eqnarray*}
  \end{proof}

  Let $ev_0 \colon \FS(X) \to X$, $c \mapsto c(0)$
  be the evaluation map at $0$.  
  It is $\GL_n(\IZ)$-equivariant, uniform  continuous  and 
  proper~\cite[Lemma~1.10]{Bartels-Lueck(2010CAT(0)flow)}. 
  Define subsets of $\FS(X)$ by
  \[
     Y(W,t)   :=   \ev^{-1}_0\bigl(X(W,t)\bigr) \quad \text{for} \quad 
                           t \geq 1, \; W \in \call',
  \]
  and set $\calv(t)  := \bigl\{Y(W,t)  \mid W \in \call' \bigr\}$,
  $|\calv(t)| :=  \bigcup_{W \in \call'} Y(W,t) \subseteq \FS(X)$.

  \begin{lemma} \label{lem:covering_infinity}
    Let $\tau,\delta > 0$ and $t \ge  1$ and set 
    $\alpha:= 4 + \delta + \tau$. 
    If $\beta > 0$ is such that for every $W \in \call'$ 
    we have $B_{\alpha}\bigl(X(W,t+ \beta)\bigr) \subseteq X(W,t)$
    then the following holds
    \begin{enumerate} 
      \item \label{lem:covering_infinity:SL_n(Z)-invariance}
         The set $\calv(t)$ is $\GL_n(\IZ)$-invariant;
      \item \label{lem:covering_infinity:calf_n}
         Each element in $\calv(t)$ is an open $\calf_n$-subset with 
         respect to the $\GL_n(\IZ)$-action;
      \item \label{lem:covering_infinity:dimension}
         The dimension of $\calv(t)$ is less or equal to $(n-2)$;
      \item \label{lem:covering_infinity:cocompact}
        The complement of $|\calv(t+\beta)|$ in $\FS(X)$ is a cocompact 
        $\GL_n(\IZ)$-subspace.
      \item \label{lem:covering_infinity:long}
        For every $c \in |\calv(t+\beta)|$ there exists $W \in \call'$ with
        \[
          B_{\delta}\bigl(\Phi_{[-\tau,\tau]}(c)\bigr) \subset Y(W,t).
        \]
    \end{enumerate}
  \end{lemma}

  The existence of a suitable $\beta$ is the assertion of 
  Lemma~\ref{lem:comparison}.

  \begin{proof}~\ref{lem:covering_infinity:SL_n(Z)-invariance},~\ref{lem:covering_infinity:calf_n}
    and~\ref{lem:covering_infinity:dimension} follow from
    Lemma~\ref{lem:covering_from_Grayson} since $ev_0$ is 
    $\GL_n ( \IZ)$-equivariant.
    \\[2pt]~\ref{lem:covering_infinity:cocompact} follows from
    Lemma~\ref{lem:Grayson}~\ref{lem:Grayson:G-compact} since the map $\ev_0
    \colon \FS(X) \to X$ is proper and
    $\ev^{-1}_0\bigl(X-|\calw(t+\beta)|\bigr) = \FS(X) - |\calv(t+\beta)|$.
    \\[2pt]
    For~\ref{lem:covering_infinity:long} consider $c \in |\calv(t+\beta)|$. 
    Choose $W \in \call'$ with $c \in Y(W,t+\beta)$.  
    Then $c(0) \in X(W,t+\beta)$.  
    Consider $d \in B_{\delta}\bigl(\Phi_{[-\tau,\tau]}(c)\bigr)$. 
    Lemma~\ref{lem:estimate_for_d(0)} implies 
    $d_X\bigl(d(0),c(0)\bigr) < \alpha$. 
    Hence $d(0) \in B_{\alpha}(c(0))$.  
    We conclude $d(0) \in B_{\alpha}\bigl(X(W,t+\beta)\bigr))$. 
    This implies $d(0) \in X(W,t)$ and hence $d \in Y(W,t)$. 
    This shows 
      $B_{\delta}\bigl(\Phi_{[-\tau,\tau]}(c)\bigr) \subseteq Y(W,t)$.
  \end{proof}

  Let $\FS_{\le \gamma}(X)$ be the subspace of $\FS(X)$ 
  of those generalized geodesics $c$  for which there exists for every 
  $\epsilon > 0$ a number $\tau \in (0, \gamma + \epsilon]$ and 
  $g \in G$ such that $g\cdot c = \Phi_{\tau}(c)$ holds.
  As an instance of~\cite[Theorem~4.2]{Bartels-Lueck(2010CAT(0)flow)}
  we obtain the following in our situation. 
  
  \begin{theorem}
      \label{the:covering_of_FS(X)_gamma}  
    There is a natural number $M$  such that for every 
    $\gamma > 0$ and every compact subset $L\subseteq X$ there exists a $\GL_n(\IZ)$-invariant 
    collection $\calu$ of subsets  of $\FS(X)$ satisfying:
    \begin{enumerate}
      \item \label{the:covering_of_FS(X)_gamma:VCyc}
        Each element $U \in \calu$ is an open
        $\VCyc$-subset of the $\GL_n(\IZ)$-space $\FS(X)$;
      \item \label{the:covering_of_FS(X)_gamma:dim}
        We have $\dim \calu \leq M$;
      \item \label{the:covering_of_FS(X)_gamma:flow}
        There is $\epsilon > 0$ with the following property:
        for $c \in \FS_{\leq \gamma}(X)$ such that $c(t) \in \GL_n(\IZ) \cdot L$ for
        some $t \in \IR$ there is $U \in \calu$ such
        that $B_{\epsilon}(\Phi_{[-\gamma,\gamma]}(c)) \subseteq U$.
    \end{enumerate}
  \end{theorem}

  \begin{definition}[{\cite[Definition~5.5]{Bartels-Lueck(2010CAT(0)flow)}} ]
    \label{def:at-infinity_plus_periods} 
    Let $G$ be a group, $\calf$ be a family of subgroups of $G$,
    $(\FS,d_{\FS})$ be a locally compact metric space with a proper 
    isometric $G$-action and $\Phi \colon \FS \times \IR \to \FS$ be a $G$-equivariant flow.

    We say that $\FS$ 
      \emph{admits long $\calf$-covers at infinity and at periodic flow lines}
    if the following holds:
   
    There is $N > 0$ such that for every $\gamma > 0$ there is a
    $G$-invariant collection of open 
    $\calf$-subsets $\calv$  of  $\FS$ and $\e > 0$  
    satisfying:
    \begin{enumerate}
         \item \label{def:at-infty_plus_periods:dim}
             $\dim \calv \leq N$;
         \item \label{def:at-infty_plus_periods:covers}
           there is a compact subset $K \subseteq \FS$ such that
           \begin{enumerate}
             \item $\FS_{\leq \gamma} \cap G \cdot K = \emptyset$;
             \item for $z \in \FS - G \cdot K$ there is 
                $V \in \calv$ such that 
                $B_\e(\Phi_{[-\gamma,\gamma]}(z)) \subset V$.
           \end{enumerate}
    \end{enumerate}
  \end{definition}

  We remark that it is natural to think of this definition as requiring 
  two conditions, the first dealing with everything outside some 
  cocompact subset (``at infinity'') and the second dealing with
  (short) periodic orbits of the flow that meet
  a given cocompact subset (``at periodic flow lines''). 
  In proving that this condition is satisfied in our situation in
  the next lemma we deal with these conditions separately.
  For the first condition we use the sets $Y(W,t)$ introduced earlier;
  for the second the theorem cited above.
  
  \begin{lemma} \label{lem:long_coverings}
    The flow space $\FS(X)$   admits long $\calf_n$-covers at infinity
    and at periodic flow lines.
  \end{lemma}

  \begin{proof}
    Fix $\gamma > 0$. Choose $t \ge   1$. 
    Put $\delta := 1$ and $\tau := \gamma$. 
    Let $\beta > 0$ be the number appearing in
    Lemma~\ref{lem:covering_infinity} and let $M\in \IN$ be the number appearing in Theorem~\ref{the:covering_of_FS(X)_gamma}. 
    Since by Lemma~\ref{lem:Grayson}~\ref{lem:Grayson:G-compact} the
    complement of $|\calw(t+\beta)|$ in $X$ is cocompact, we can find a compact
    subset $L$ of this complement such that 
    \[ 
        \GL_n(\IZ) \cdot L = X \setminus |\calw(t+\beta)|.
    \]
    For this compact subset $L$ we obtain a real number
    $\epsilon > 0$ and a set $\calu$ of subsets of $\FS(X)$ from
    Theorem~\ref{the:covering_of_FS(X)_gamma}. 
    We can arrange that $\epsilon \le 1$.

    Consider $\calv := \calu \cup \calv(t)$, where $\calv(t)$ is the
    collection of open subsets defined before Lemma~\ref{lem:covering_infinity}.
    We want to show that $\calv$ satisfies
    the conditions appearing in Definition~\ref{def:at-infinity_plus_periods} 
    with respect to the number $N := M +n - 1$.

    Since the covering dimension of $\calu$ is less or equal to $M$ by
    Theorem~\ref{the:covering_of_FS(X)_gamma}~\ref{the:covering_of_FS(X)_gamma:dim}
    and the covering dimension of $\calv(t)$ is less or equal to $n-2$ by
    Lemma~\ref{lem:covering_infinity}~\ref{lem:covering_infinity:dimension}, the
    covering dimension of $\calu \cup \calv(t)$ is less or equal to $N$.

    Since $\calu$ and $\calv(t)$ are $\GL_n(\IZ)$-invariant by
    Theorem~\ref{the:covering_of_FS(X)_gamma} and
    Lemma~\ref{lem:covering_infinity}~\ref{lem:covering_infinity:SL_n(Z)-invariance},
    $\calu \cup \calv(t)$ is $\GL_n(\IZ)$-invariant.

    Since each element of $\calu$ is an open $\VCyc$-set by
    Theorem~\ref{the:covering_of_FS(X)_gamma}~\ref{the:covering_of_FS(X)_gamma:VCyc}
    and each element of $\calv(t)$ is an open $\calf_n$-subset by
    Lemma~\ref{lem:covering_infinity}~\ref{lem:covering_infinity:calf_n}, each
    element of $\calu \cup \calv(t)$ is an open $\calf_n$-subset, as 
    $\VCyc \subset \calf_n$.
    Define
    \[
     S := \{c \in \FS(X) \mid \exists Z \in \calu \cup \calv(t) \; \text{with}\;
     \overline{B}_{\epsilon}\bigl(\Phi_{[-\gamma,\gamma]}(c)\bigr) \subseteq Z\}.   
    \]
    This set $S$ contains $\FS(X)_{\le \gamma} \cup |\calv(t+\beta)|$ 
    by the following argument.  
    If $c \in |\calv(t+\beta)|$, then $c \in S$ by
    Lemma~\ref{lem:covering_infinity}~\ref{lem:covering_infinity:long}.  
    If $c \in \FS(X)_{\le \gamma}$ and $c \notin |\calv(t+\beta)|$, 
    then $c \in \FS(X)_{\le \gamma}$ and $c(0) \in \GL_n(\IZ) \cdot L$ 
    and hence $c \in S$ by
    Theorem~\ref{the:covering_of_FS(X)_gamma}~\ref{the:covering_of_FS(X)_gamma:flow}.

    The subset $S \subseteq \FS(X)$ is $\GL_n(\IZ)$-invariant because
    $\calu \cup \calv(t)$ is $\GL_n(\IZ)$-invariant.
    
    Next we prove that $S$ is open. 
    Assume that this is not the case. 
    Then there exists $c \in S$ and a sequence $(c_k)_{k \ge 1}$ of elements in 
    $\FS(X) - S$ such that $d_{\FS(X)}\bigl(c,c_k\bigr) <1/k$ 
    holds for $k \ge 1$. 
    Choose $Z \in \calu \cup \calv(t)$ with
    $\overline{B}_{\epsilon}\bigl(\Phi_{[-\gamma,\gamma]}(c)\bigr) \subseteq Z$.
    Since $\FS(X)$ is proper as metric space
    by~\cite[Proposition~1.9]{Bartels-Lueck(2010CAT(0)flow)} and
    $\overline{B}_{\epsilon}\bigl(\Phi_{[-\gamma,\gamma]}(c)\bigr)$ has bounded
    diameter, $\overline{B}_{\epsilon}\bigl(\Phi_{[-\gamma,\gamma]}(c)\bigr)$ is
    compact.  
    Hence we can find $\mu > 0$ with 
    $B_{\epsilon+  \mu}\bigl(\Phi_{[-\gamma,\gamma]}(c)\bigr) \subseteq Z$.  
    We conclude from~\cite[Lemma~1.3]{Bartels-Lueck(2010CAT(0)flow)} for all 
    $s \in [-\gamma,\gamma]$
    \[
       d_{\FS(X)}\bigl(\Phi_s(c),\Phi_s(c_k)\bigr) \le e^{|s|} \cdot
       d_{\FS(X)}\bigl(c,c_k\bigr) < e^{\tau} \cdot 1/k.
    \]
    Hence we get for $k \ge 1$
    \[
      B_{\epsilon}\bigl(\Phi_{[-\gamma,\gamma]}(c_k)\bigr) \subseteq B_{\epsilon+
          e^{\tau} \cdot 1/k}\bigl(\Phi_{[-\gamma,\gamma]}(c)\bigr)
    \]
    Since $c_k$ does not belong to $S$, we conclude that
    $B_{\epsilon+  e^{\tau} \cdot 1/k}\bigl(\Phi_{[-\gamma,\gamma]}(c)\bigr)$ 
    is not contained in $Z$.
    This implies $e^{\tau} \cdot 1/k \ge \mu$ for all $k \ge 1$, a contradiction.
 
    The $\GL_n(\IZ)$-set $\FS(X) - |\calv(t+\beta)|$ is cocompact by
    Lemma~\ref{lem:covering_infinity}~\ref{lem:covering_infinity:cocompact}.  
    Since $S$ is an open $\GL_n(\IZ)$-subset of $\FS(X)$ and 
    contains $|\calv(t+\beta)|$, the $\GL_n(\IZ)$-set $\FS(X) - S$ is cocompact.  
    Hence we can find a compact subset $K \subseteq \FS(X) - S$ satisfying
    \[
     \GL_n(\IZ) \cdot K = \FS(X) - S.
    \]
    Obviously $\FS(X)_{\le \gamma} \cap \GL_n(\IZ) \cdot K  = \emptyset$.
  \end{proof}

  \begin{proof}[Proof of Theorem~\ref{the:transfer-reducible}]
    The group $\GL_n(\IZ)$ and
    the associated flow space $\FS(X)$
    satisfy~\cite[Convention~5.1]{Bartels-Lueck(2010CAT(0)flow)} by the
    argument of~\cite[Section~6.2]{Bartels-Lueck(2010CAT(0)flow)}. 
    Notice that in \cite[Convention~5.1]{Bartels-Lueck(2010CAT(0)flow)} 
    it is \emph{not} required that the action is cocompact.  
    The argument of~\cite[Section~6.2]{Bartels-Lueck(2010CAT(0)flow)}
    showing that there 
    is a constant $k_G$ such that the order of any finite subgroup of $G$ is
    bounded by $k_G$ uses that the action is cocompact.  
    But such a number exists for $\GL_n(\IZ)$ as well, since $\GL_n(\IZ)$ is
    virtually torsion free (see~\cite[Exercise II.3 on p.~41]{Brown(1982)}).

    Because of~\cite[Proposition~5.11]{Bartels-Lueck(2010CAT(0)flow)} it
    suffices to show that $\FS(X)$ admits long $\calf_n$-covers at
    infinity and at periodic flow lines in the sense of
    Definition~\ref{def:at-infinity_plus_periods} and admits contracting
    transfers in the sense of~\cite[Definition~5.9]{Bartels-Lueck(2010CAT(0)flow)}.
    For the first condition this has been
    done in Lemma~\ref {lem:long_coverings},
    while the second condition follows from the argument given
    in~\cite[Section~6.4]{Bartels-Lueck(2010CAT(0)flow)}.
  \end{proof}

  \begin{proposition} \label{prop:FJ-for-GL_Z-K-theory-leq-1}
     The  $K$-theoretic FJC up to dimension $1$ holds for $\GL_n(\IZ)$.
  \end{proposition}
  
  \begin{proof}
    We proceed by induction over $n$.
    As $\GL_1(\IZ)$ is finite, the initial step of the induction is trivial.

    Since $\GL_n(\IZ)$ is transfer reducible over $\calf_n$ by 
    Theorem~\ref{the:transfer-reducible} it follows 
    from~\cite[Theorem~1.1]{Bartels-Lueck(2012annals)} that
    $\GL_n(\IZ)$ satisfies the  $K$-theoretic FJC 
    up to dimension $1$ with respect to $\calf_n$.
    It remains to replace $\calf_n$ by the family $\VCyc$.
    Because of the Transitivity Principle~\ref{prop:trans-princ} 
    it suffices to show that each $H \in \calf_n$ satisfies 
    the FJC up to dimension $1$
    (with respect to $\VCyc$).
    Combining the induction assumption for $\GL_k(\IZ)$, $k < n$
    with well known inheritance properties for direct products, 
    exact sequences of groups and subgroups 
    (see for example~\cite[Theorem~1.10, Corollary~1.13, Theorem~1.9]
          {Bartels-Farrell-Lueck(2011cocomlat)})
    it is easy to reduce the $K$-theoretic 
    FJC up to dimension $1$ for members of $\calf_n$ to
    the class of virtually poly-cyclic groups.
    Finally, for virtually poly-cyclic groups the FJC
    holds by~\cite{Bartels-Farrell-Lueck(2011cocomlat)}.
  \end{proof}

%======================================================

\section{Strong Transfer Reducibility of \texorpdfstring{$\GL_n(\IZ)$}{GLnZ}}
   \label{sec:Strong-Trans-Red-GLnZ}

  In this section we will discuss the modifications
  needed to extend Proposition~\ref{prop:FJ-for-GL_Z-K-theory-leq-1} 
  to higher $K$-theory. 
  The necessary tools for this extension  
  have been developed by Wegner~\cite{Wegner(2012FJ_CAT0)}.

  \begin{theorem} \label{thm:GLnZ-strongly-trans-red}
     The group $\GL_n(\IZ)$ is strongly transfer reducible over $\calf_n$
     in the sense of~\cite[Definition~3.1]{Wegner(2012FJ_CAT0)}.
  \end{theorem}

  Wegner proves in~\cite[Theorem~3.4]{Wegner(2012FJ_CAT0)} that $\CAT(0)$-groups
  are strongly transfer reducible over $\VCyc$.
  As $\GL_n(\IZ)$ does not act cocompactly on $X$, we cannot use Wegner's result
  directly.
  However, in combination with Lemma~\ref{lem:long_coverings}
  his method yields a proof of Theorem~\ref{thm:GLnZ-strongly-trans-red}.

  \begin{proof}[Proof of Theorem~\ref{thm:GLnZ-strongly-trans-red}]
    The only place where Wegner uses cocompactness of the action is
    when he verifies the assumptions 
    of~\cite[Theorem~5.7]{Bartels-Lueck(2010CAT(0)flow)}, 
    see~\cite[Proof of Theorem~3.4]{Wegner(2012FJ_CAT0)}. 
  
    We know by Lemma~\ref{lem:long_coverings} that $FS(X)$  admits 
    long $\calf_n$-covers at infinity and periodic flow lines. 
    Wegner cites~\cite[Subsection~6.3]{Bartels-Lueck(2010CAT(0)flow)} 
    for this assumption. In the cocompact setting the family can even be chosen to be $\VCyc$.
  
    That the assumptions of \cite[Convention~5.1]{Bartels-Lueck(2010CAT(0)flow)}, which are used
    implicitly in \cite[Theorem~5.7]{Bartels-Lueck(2010CAT(0)flow)}, are satisfied has already been
    explained in the proof of Theorem~\ref{the:transfer-reducible}.
  \end{proof}
  
  \begin{theorem} \label{thm:K-FJ-for-GLnZ}
     The $K$-theoretic FJC holds for $\GL_n(\IZ)$.
  \end{theorem}

  \begin{proof}
    Theorem~\ref{thm:GLnZ-strongly-trans-red} together 
    with~\cite[Theorem~1.1]{Wegner(2012FJ_CAT0)} imply that
    $\GL_n(\IZ)$ satisfies the $K$-theoretic FJC
    with respect to the family $\calf_n$.

    Using the induction from the proof of 
    Proposition~\ref{prop:FJ-for-GL_Z-K-theory-leq-1}
    the family $\calf_n$ can be replaced by $\VCyc$.
  \end{proof}
  
%======================================================

\section{Wreath products and transfer reducibility}
   \label{sec:wreath-product_and_transfer-reducibility}

  Our main result in this section is the following
  variation of~\cite[Theorem~1.1]{Bartels-Lueck(2012annals)}.

  \begin{theorem} \label{thm:transfer-red_implies_FICwF}
    Let $\calf$ be a family of subgroups of the group $G$ and let $F$ be a  finite group.
     Denote by $\calf^\wr$ the family of subgroups $H$ of $G \wr F$ that contain
    a subgroup of finite index  that is isomorphic to  
    a subgroup of  $H_1 \x \dots \x H_n$
    for some $n$ and $H_1,\dots,H_n \in \calf$.

   \begin{enumerate} 
    \item \label{thm:transfer-red_implies_FICwF:normal}
    If $G$ is  transfer reducible over $\calf$,
    then  the wreath product $G \wr F$ 
    satisfies the  $K$-theoretic FJC up
    to dimension $1$ with respect to $\calf^\wr$ and
    the $L$-theoretic FJC with respect to $\calf^\wr$;

    \item \label{thm:transfer-red_implies_FICwF:strongly}
    If $G$ is strongly transfer reducible over $\calf$, then $G\wr F$ 
    satisfies the  $K$-theoretic and $L$-theoretic FJC in all dimensions
     with respect to $\calf^\wr$.
   \end{enumerate}
 \end{theorem}

  The idea of the proof of this result is very easy.
  We only need to show that $G \wr F$ is transfer reducible 
  over $\calf^\wr$ and apply~\cite[Theorem~1.1]{Bartels-Lueck(2012annals)}.
  However, it will be easier to verify a slightly weaker condition for $G \wr F$.
 
  \begin{definition}[Homotopy $S$-action; 
         {\cite[Definition~1.4]{Bartels-Lueck(2012annals)}}]
       \label{def:S-action_plus_long-covers}
    Let $S$ be a finite subset of a group $G$ (containing
    the identity element $e \in G$).
    Let $X$ be a space.
    \begin{enumerate}
    \item \label{def:S-action_plus_long-covers:action}
      A \emph{homotopy $S$-action $(\varphi,H)$ on $X$} consists of
      continuous maps $\varphi_g \colon X \to X$
      for $g \in S$ and homotopies
      $H_{g,h} \colon X \times [0,1] \to X$
      for $g,h \in S$ with $gh \in S$ such that
      $H_{g,h}(-,0) = \varphi_g \circ \varphi_h$
      and $H_{g,h}(-,1) = \varphi_{gh}$ holds for $g,h \in S$ with $gh \in S$.
      Moreover, it is  required that $H_{e,e}(-,t) = \varphi_e = \id_X$
      for all $t \in [0,1]$;
     \item \label{def:S-action_plus_long-covers:F}
      For $g \in S$ let $F_g(\varphi,H)$ be the set of all
      maps $X \to X$ of the form $x \mapsto H_{r,s}(x,t)$
      where $t \in [0,1]$ and $r,s \in S$ with $rs = g$;
     \item \label{def:S-action_plus_long-covers:S(g,x)}
      Given  a subset $A\subset G\times X$ let $S^1_{\varphi,H}(A)
      \subset G \times X$ denote the set
      \begin{multline*}
       \hspace{10mm} \{(g a^{-1} b,y)\mid  \exists \; x \in X, \;  a,b\in
       S, \; f \in F_{a}(\varphi,H), \; \tilde{f} \in F_{b}(\varphi,H)
       \\
       \text{satisfying} \; (g,x) \in A, \; f(x) = \tilde{f}(y) \}.
     \end{multline*}
      Then define inductively $S^n_{\varphi,H}(A):=S^1_{\varphi,H}(S^{n-1}_{\varphi,H}(A))$; 
     \item \label{def:S-action_plus_long-covers:long}
      Let $(\varphi,H)$ be a homotopy $S$-action on $X$
      and $\calu$ be an open cover of $G \times X$.
      We say that $\calu$ is \emph{$S$-long with respect to $(\varphi,H)$}
      if for every $(g,x) \in G \times X$ there is $U \in \calu$ containing
      $S^{|S|}_{\varphi,H}(g,x)$ where $|S|$ is the cardinality of $S$.
    \end{enumerate}
  \end{definition}

  We will use the following variant of~\cite[Definition~1.8]{Bartels-Lueck(2012annals)}. 
  \begin{definition}[Almost transfer reducible] 
     \label{def:almost_transfer_reducible}
    Let $G$ be a group and $\calf$ be a family of subgroups.
    We will say that $G$ is \emph{almost transfer reducible} over $\calf$ if
    there is a number $N$ such that for any
    finite subset $S$ of $G$ we can find
    \begin{enumerate}
     \item a contractible, compact, controlled $N$-dominated,
      metric space $X$ (\cite[Definition~0.2]{Bartels-Lueck(2010CAT(0)flow)}),
      equipped with a homotopy $S$-action $(\varphi,H)$ and
      \item a $G$-invariant cover $\calu$ of $G \times X$ of dimension
      at most $N$ that is $S$-long. 
      Moreover we require that for all $U \in \calu$ the subgroup
      $G_U := \{ g \in G \mid gU = U \}$ belongs to $\calf$.
      (Here we use the $G$-action on $G \x X$  given 
         by $g \cdot (h,x) = (gh,x)$.)
    \end{enumerate}
  \end{definition}
  
  The original definition for transfer reducibility requires in addition
  that $gU$ and  $U$ are disjoint if $U \in \calu$ and $g \notin G_U$,
  in other words each $U$ is required to be an $\calf$-subset.
  One can also drop this condition from the 
  notion of ``strongly transfer reducible'' 
  introduced in~\cite[Definition~3.1]{Wegner(2012FJ_CAT0)}. 
  A group satisfying this weaker version will be called 
  \emph{almost strongly transfer reducible}.

  The result corresponding to~\cite[Theorem~1.1]{Bartels-Lueck(2012annals)}  
  (respectively \cite[Theorem~1.1]{Wegner(2012FJ_CAT0)})
  is as follows.

  \begin{proposition} \label{prop:transfer-red'_implies_FJC} 
    Let $\calf$ be a family of subgroups of a group $G$ and let $\calf'$  
    be the family of subgroups of $G$ that contain 
    a member of $\calf$ as a finite index subgroup.
  \begin{enumerate}
  \item \label{prop:transfer-red'_implies_FJC:normal}  
   If $G$ is almost transfer reducible over $\calf$,
    it satisfies the  $K$-theoretic FJC up
    to dimension $1$ with respect to $\calf'$ and
    the $L$-theoretic FJC with respect to $\calf'$.
  \item \label{prop:transfer-red'_implies_FJC:strongly}  
   If $G$ is almost strongly transfer reducible
    over $\calf$, it satisfies the  $K$-theoretic FJC 
    with respect to $\calf'$ and
    the $L$-theoretic FJC with respect to $\calf'$.
  \end{enumerate}
  \end{proposition}
  
  \begin{proof}~\ref{prop:transfer-red'_implies_FJC:normal}
    The proof can be copied almost word for word from the 
    proof of~\cite[Theorem~1.1]{Bartels-Lueck(2012annals)}.
    The only difference is that we no longer know that the isotropy
    groups of the action of $G$ on the geometric realization $\Sigma$ of
    the nerve of $\calu$ belong to $\calf$,  and this action may no longer
    be cell preserving. 
   But it is still simplicial and therefore we can just replace $\Sigma$ 
   by its barycentric subdivision by the following 
   Lemma~\ref{replacing_by_barycentric_subdivision} provided  
   we replace $\calf$ by $\calf'$.
    (The precise place where this makes a difference is in the proof
    of~\cite[Proposition~3.9]{Bartels-Lueck(2012annals)}.)
   \\[1ex]~\ref{prop:transfer-red'_implies_FJC:strongly} 
   The L-theory part follows from 
   part~\ref{prop:transfer-red'_implies_FJC:normal}   
   since almost strongly transfer reducible implies almost transfer reducible. 

   For the proof of the $K$-theory part we have to adapt Wegner's proof
   of~\cite[Theorem~1.1]{Wegner(2012FJ_CAT0)}.
   The necessary changes concern~\cite[Proposition~3.6]{Wegner(2012FJ_CAT0)}
   and  are  similar to the changes discussed above,
   but as Wegner's argument is somewhat differently organized
   we have to be a little more careful.
   First we observe that for any fixed $M > 0$ the second assertion
   in~\cite[Proposition~3.6]{Wegner(2012FJ_CAT0)} can be strengthened to
   \begin{equation} \label{eq:add-eps-to-Wegners-3.6}
    M k \cdot d_\Sigma^1(f(g,x),f(h,y)) \leq 
            d_{\Psi,S,k,\Lambda}( (g,x), (h,y) )  \; \text{for all}
        (g,x),(h,y) \in G \x X.
   \end{equation} 
   To do so we only need to set $n := M \cdot 4Nk$
   instead of $4Nk$ in the second line of Wegner's proof;
   then $k$ can be replaced by $M \cdot k$ in the denominator
   of the final expression on p.786.
   (Of course $X$, $\Psi$, $\Lambda$, $\Sigma$ and $f$
   depend now also on $M$.)
   Then we can replace $\Sigma$ by its barycentric subdivision $\Sigma'$.
   Using Lemma~\ref{replacing_by_barycentric_subdivision}
   we conclude from~\eqref{eq:add-eps-to-Wegners-3.6} with
   sufficiently large $M > 0$ that
   \begin{itemize}
   \item $d^1_{\Sigma'}(f(g,x),f(h,y)) \leq \frac{1}{k}$ for all
        $(g,x), (h,y) \in G \x X$ satisfying the inequality
        $d_{\Psi,S,k,\Lambda}( (g,x), (h,y) ) \leq k$.
   \end{itemize}
   This assertion still guarantees that the maps $(f_n)$ induce a 
   functor as needed on the right hand side of the diagram 
   on p.789 of~\cite{Wegner(2012FJ_CAT0)}. 
   With this change Wegner's argument proves the $K$-theory
   part of~\ref{prop:transfer-red'_implies_FJC:strongly}.
  
   (Alternatively one can strengthen 
   Lemma~\ref{replacing_by_barycentric_subdivision} below and 
   check that  the $l^1$-metric
   under barycentric subdivision changes only up to
   Lipschitz equivalence.
   Thus~\cite[Proposition~3.6]{Wegner(2012FJ_CAT0)} remains in fact true
   for $\Sigma'$.)  
  \end{proof}

  \begin{lemma} \label{replacing_by_barycentric_subdivision} 
    Let $G$ be a group,
    and let $\calf $ be the family of subgroups of $G$.  Let $\calf'$ be the
    family of subgroups of $G$ that contain a member of $\calf$ as a finite
    index subgroup.  Let $\Sigma$ be a simplicial complex with a simplicial
    $G$-action such that the isotropy group of each vertex is contained in
    $\calf$. Let $\Sigma'$ be the barycentric subdivision. Denote by
    $d^1_{\Sigma}$ the $l^1$-metric on $\Sigma$ and by $d^1_{\Sigma'}$ the
    $l^1$-metric on $\Sigma'$.
    \begin{enumerate}
    \item \label{replacing_by_barycentric_subdivision:isotropy}
      The group $G$ acts cell preserving on $\Sigma'$. 
      All isotropy groups of $\Sigma'$ lie in $\calf'$.
      In particular $\Sigma'$ is a $G$-$CW$-complex whose isotropy 
      groups belong to $\calf'$;

    \item \label{replacing_by_barycentric_subdivision:l1-metric} 
       Given a number $\epsilon' > 0$ and a natural number $N$, 
       there exists a number $\epsilon > 0$ depending only on 
       $\epsilon'$ and $N$ such that the following holds: If
       $\dim(\Sigma) \le N$ and $x,y \in \Sigma$ satisfy 
       $d^1_{\Sigma}(x,y) < \epsilon$, then $d^1_{\Sigma'}(x,y) < \epsilon'$.
    \end{enumerate}
  \end{lemma}

  \begin{proof} The elementary proof of
   assertion~\ref{replacing_by_barycentric_subdivision:isotropy} 
   is left to the reader.  
   The proof of
   assertion~\ref{replacing_by_barycentric_subdivision:l1-metric} is an obvious
   variation of the proof of~\cite[Lemma~9.4~(ii)]{Bartels-Lueck(2012annals)}.
   Namely, any simplicial complex can be equipped with the $l^1$-metric.  With
   those metrics the inclusion of a subcomplex is an isometric embedding.  Let
   $\hat{\Sigma}$ be the simplicial complex with the same vertices as $\Sigma$,
   such that any finite subset of the vertices spans a simplex.  The inclusions
   $i\colon \Sigma \hookrightarrow \hat{\Sigma}$ and $i'\colon \Sigma'
   \hookrightarrow \hat{\Sigma}'$ are isometric embeddings.  The images
   $i(x),i(y)$ of any two points $x,y\in \Sigma$ are contained in a closed
   simplex of dimension at most $2N+1$.  Thus it suffices to consider the case
   where $\Sigma$ is replaced by the standard $(2N+1)$-simplex.  A compactness
   argument gives the result in that case.
  \end{proof}

  \begin{remark}
    Often the family $\calf$ is closed under finite overgroups.
    In this case $\calf = \calf'$ and it is really easier to
    work with the weaker notion of almost transfer reducible instead of
    transfer reducible.
    However, there are situations in which a family that is not
    closed under finite overgroups is important.
    For example, in~\cite{Davis-Quinn-Reich(2011)} the family
    of virtually cyclic groups of type $I$ is considered. 
  \end{remark}
 
  Let $G$ be a group and $F$ be finite group.
  We will think of elements of the $F$-fold product $G^F$ 
  as functions $g \colon F \to G$.
  For $a \in F$, $g \in G^F$ we write $l_a(g) \in G^F$ for the function
  $b \mapsto g(ba)$; this defines a left action of $F$ on $G^F$ and the
  corresponding semi-direct product is the wreath product $G \wr F$
  with multiplication $ga g' a' = g l_a(g') a a'$ for $g$, $g' \in
  G^F$ and $a$, $a' \in F$.
  If $G$ acts on a set $X$, then we obtain an action
  of $G \wr F$ on $X^F$. 
  In formulas this action is given by  
  \begin{eqnarray*}
    (g \cdot x)(b) & := & g(b) \cdot x(b); \\ 
    (a \cdot x)(b) & := & x(ba),
  \end{eqnarray*}
  and hence
\[
 (ga \cdot x)(b) = g(b) \cdot x (ba)
\]
  for $g \in G^F$, $x \in X^F$ and $a,b \in F$.
  We will sometimes also write $l_a(x)$ for $a \cdot x$.
  Let now $S \subset G$ and $(\varphi,H)$ be a homotopy $S$-action on
  a space $X$.
  Set $S \wr F := \{ s a \mid s \in S^F, a \in F \} \subset G \wr F$.  
  Then we obtain a homotopy $S \wr F$-action $(\hat \varphi, \hat H)$ 
  on $X^F$.
  In formulas this is given by
  \begin{eqnarray*}
     [ \hat \varphi_{sa} (x) ] \, (b) & := & \varphi_{s(b)} (x(ba)); \\ {}
     [ \hat H_{sa,s'a'}(x,t) ]  \, (b) & := & H_{s(b),s'(ba)} (x(baa'),t), 
  \end{eqnarray*}
  for $t \in [0,1]$, $s,s' \in S^F$, $a,a',b \in F$ with 
  $sas'a' = s \, l_a(s') a a' \in S\wr F$. 
Hence $(s \, l_a(s')(b)=s(b)s'(ba)\in S$ and the right hand side is defined.
  It is easy to check that if $X$ is a contractible, compact, controlled 
  $N$-dominated, metric space, then the same is true for $X^F$ provided we
  replace $N$ by $N \cdot |F|$.
  \begin{proof}[Proof of Theorem~\ref{thm:transfer-red_implies_FICwF}]
    Let us postpone the ``strong''-case until the end of the proof.
     Because of Proposition~\ref{prop:transfer-red'_implies_FJC} it suffices to
     show that $G \wr F$ is almost transfer reducible over $\calf^\wr$.      
     Let $\hat S$ be a finite subset  of $G\wr F$. 
     By enlarging it we can assume that it has the form
     $S \wr F$ for some finite subset $S\subset G$. 
     Pick a finite subset $S'\subset G$ such that 
     $S\subset S'$ and $|S'|\ge |S\wr F|$. 
     (If $G$ is finite, then $G \wr F \in \calf^\wr$ and there is
     nothing to prove.)
     As $G$ is transfer reducible and hence in particular almost transfer reducible, there is a number $N$, 
     (depending only on $G$, not on $\hat S$ or $S$)
     a compact, contractible, controlled $N$-dominated, metric space $X$,
     a homotopy $S'$-action $(\varphi,H)$ on $X$, and a $G$-invariant 
     $S'$-long open cover $\calu$ of $G \x X$ of dimension at most $N$ 
     such that for all $U \in \calu$ we have $G_U \in \calf$.
     As pointed out before, $X^F$ is a compact, contractible, 
     controlled $N \cdot |F|$-dominated, metric space. 
     For $u \colon F \to \calu$, let 
     $V^u := \{ (g,x) \in G^F \x X^F \mid 
                                            (g(b),x(b)) \in u(b) \; \text{for all} \; b \in F \}$.
     We obtain an open cover 
     $\calv := \{ V^u \mid u \colon F \to \calu \}$ of $G^F \x X^F$
     of dimension at most $(N+1)^{|F|} - 1$. 
     
     This cover is invariant for the $G \wr F$-action defined by
     \begin{eqnarray*}
        [g \cdot ( h ,x)] \, (b) & := & (g(b)h(b),x(b));\\ {}
        [a \cdot ( h ,x)] \, (b) & := & (h(ba),x(ba)), 
     \end{eqnarray*}
     for $g,h \in G^F$, $x \in X^F$ and $a,b \in F$.
     As we have $(G^F)_{V^u} = \prod_{b \in F} G_{u(b)}$,
     it follows that $(G \wr F)_{V} \in \calf^\wr$ for all $V \in \calv$.
     Now we pull back $\calv$ to a cover 
     $\hat \calu := \{ p^{-1}(V) \mid V \in \calv \}$ of $G \wr F \x X^F$
     along the $G \wr F$-equivariant map 
     $p \colon G \wr F \x X^F \to G^F \x X^F$, $(ga,x) \mapsto (g,a
     \cdot x)$. Here $G \wr F$ operates on $G \wr F \times X^F$ via
     left multiplication on the first factor and on $G^F \times X^F$
     via the operation defined above.

     The definition of the homotopy $S\wr F$-action on $X^F$ gives 
     \[ 
     \hat H_{sa,s'a'}(-,t) = 
      \hat H_{s,l_{a}(s')} (-,t) \circ l_{aa'} = l_{aa'}
       \circ \hat H_{l_{{a'}^{-1}a^{-1}}(s),l_{{a'}^{-1}}(s')} (-,t)\]
     for $s,s'\in S^F$, $a,a'\in F$. 
     For $\bar s := s \; l_a(s')$, $\bar a := a a'$ we have 
     $\bar s \bar a = s a s' a'$ and consequently 
     \begin{equation} \label{eq:F}
     F_{\bar s \bar a}(\hat \varphi, \hat H) 
      = F_{\bar s}(\hat \varphi|_{G^F}, \hat H|_{G^F}) \circ l_{\bar a} =
      l_{\bar a} \circ 
       F_{l_{\bar a^{-1}}(\bar s)}(\hat \varphi|_{G^F}, \hat H|_{G^F}).
     \end{equation}
     Let us insert this into the definition of 
     $(S\wr F)^1_{\hat \varphi,\hat H}(hb,x)$ 
     with $h\in G^F$, $b\in F$, $x \in X^F$. 

     Pick any $(h'b',x')\in (S\wr F)^1_{\hat \varphi,\hat H}(hb,x)$ 
     with $h'\in G^F$, $b'\in F$. 
     Then there are elements
     $s,s'\in S^F$, $a,a'\in F$ and 
     $f\in F_{sa}(\hat \varphi, \hat H)$,
     $\tilde{f}\in F_{s'a'}(\hat \varphi, \hat H)$ such that 
     \[
     f(x) = \tilde{f}(x'), \quad \text{and} \quad
     h'b'=hb(sa)^{-1}s'a' 
     %= h \, l_{ba^{-1}}(s^{-1}) \, l_{ba^{-1}}(s') ba^{-1}a'.
     \]    
     Using the first equality in~\eqref{eq:F}  we find 
     $\bar f \in F_{s}(\hat \varphi|_{G^F},\hat H|_{G^F})$
     such that $f = \bar f \circ l_a$.
     Using the second equality in~\eqref{eq:F}  we find
     $f' \in F_{l_{ba^{-1}}(s)}(\hat \varphi|_{G^F},\hat H|_{G^F})$
     such that $\bar f \circ l_{ab^{-1}} = l_{ab^{-1}} \circ f'$.
     Then $f \circ l_{b^{-1}} = l_{ab^{-1}} \circ f'$; equivalently
     $l_{ba^{-1}} \circ f = f' \circ l_b$.
     Similarly, using both equations in~\eqref{eq:F} again,
     we find
     $\tilde f' \in F_{l_{ba^{-1}}(s')}(\hat \varphi|_{G^F},\hat H|_{G^F})$
     such that $l_{ba^{-1}} \circ \tilde f = \tilde f' \circ l_{ba^{-1}a'}$.
     
     We claim that $p(h'b',x')$ belongs to 
     $(S^F)^1_{\hat \varphi|_{G^F},\hat H|_{G^F}}(p(hb,x))$.  
     We have $p(h'b',x') = (h',b'\cdot x')$ and
     $p(hb,x) = (h,b \cdot x)$.   
     From $h'b'=hb(sa)^{-1}s'a'$ we conclude
     $h' = h \, l_{b a^{-1}} ( s^{-1} s' )$ and $b' = b a^{-1} a'$.
     Now the equations
     \begin{eqnarray*}
       f'(b \cdot x) & = & l_{ba^{-1}}(f(x)) = l_{ba^{-1}} (\tilde f (x') )
           = \tilde f' (b a^{-1} a' \cdot x' ) = \tilde f' ( b' x') \\
       h'  & = & h l_{ba^{-1}}(s^{-1} s') 
             = h (l_{ba^{-1}}(s))^{-1} l_{ba^{-1}}(s')  
     \end{eqnarray*}
     prove our claim.

     Summarizing  we have shown that for any $A\subset G\wr F\times X^F$
     we have
     \[
     p((S\wr F)^1_{\hat \varphi,\hat H}(A)) 
         \subset (S^F)^1_{\hat \varphi|_{G^F},\hat H|_{G^F}}(p(A)).
     \]
     By induction then for all $n$
     \[
     p((S\wr F)^{n}_{\hat \varphi,\hat H}(A)) 
         \subset (S^F)^{n}_{\hat \varphi|_{G^F},\hat H|_{G^F}}(p(A)).
     \]
     Since the $S^F$-homotopy action on $X^F$ is defined componentwise we 
     have 
     \[
      (S^F)^n_{\hat \varphi|_{G^K},\hat H|_{G^K}}(g,x) 
                \subset \prod_{a\in F} S^n_{\varphi,H}(g(a),x(a)).
     \]
     Recall the definition of $S'$ from the beginning of the proof.
     Since the cover $\calu$ of $G\times X$ is $S'$-long, there is 
     for each $a\in F$ a $u(a)\in \calu$ with 
     $(S')^{|S'|}_{\varphi,H}(h(a),x(ab))\subset u(a)$. 
     Thus we obtain
     \begin{eqnarray*} 
      p((S\wr F)^{|S\wr F|}_{\hat \varphi,\hat H}(hb,x))
      &\subset & 
      (S^F)^{|S\wr F|}_{\hat \varphi|_{G^F},\hat H|_{G^F}}(h,b\cdot x) 
          \subset \prod_{a\in F} S^{|S\wr F|}_{\varphi,H}(h(a),x(ab))
      \\ 
      &\subset& 
      \prod_{a\in F} S'^{|S'|}_{\varphi,H}(h(a),x(ab)) \subset \prod_{a\in F} u(a)        = V^u.
     \end{eqnarray*}
     So $(S\wr F)^{|S\wr F|}_{\hat \varphi,\hat H}(hb,x)\subset p^{-1}(V^u)$. 
     Hence the cover $\hat \calu$ is $S\wr F$-long. 
      
      If $X$ is equipped with a strong homotopy action $\Psi$ 
      (see \cite[Section~2]{Wegner(2012FJ_CAT0)}),
      we obtain a strong homotopy action $\hat \Psi$ on $X^F$. 
      In formulas it is given by
      \begin{multline*}
      \hat \Psi(g_na_n,t_n,\ldots,t_1,g_0a_0,x)(b)
      \\
      \quad  \quad := 
      \Psi(g_n(b),t_n,g_{n-1}(ba_n),t_{n-1},g_{n-2}(ba_na_{n-1}), \ldots,g_0(ba_n\ldots a_1),x(ba_n\ldots a_0))
    \end{multline*}
     for $g_0,\ldots, g_n\in G^F, a_0,\ldots,a_n,b\in F,t_1,\ldots, t_n\in [0,1],x\in X^F$. We also have here 
      \begin{eqnarray*}
     \lefteqn{\hat \Psi(g_na_n,t_n,\ldots,t_1,g_0a_0,-)}
    & &
    \\
     &=& 
    \hat \Psi(g_n,t_n,l_{a_n}(g_{n-1}),t_{n-1},l_{a_na_{n-1}}(g_{n-2}),t_{n-2},\ldots,l_{a_n\ldots a_1}(g_0),-)\circ l_{a_n\ldots a_0}
     \\
    &=& 
    l_{a_n\ldots a_0} \circ \hat \Psi(l_{(a_n\ldots a_0)^{-1}}g_n,t_n,l_{(a_{n-1}\ldots a_0)^{-1}}(g_{n-1}),t_{n-1},\ldots,\ignore{l_{(a_1a_0)^{-1}}(g_1),t_1,}l_{a_0^{-1}}(g_0),-).
      \end{eqnarray*}
	For $\bar g:= g_nl_{a_n}(g_{n-1})l_{a_na_{n-1}}(g_{n-2})\ldots
        l_{a_n\ldots a_1}(g_0) \in G^F,\bar a := a_n\ldots a_0\in
        F,n\in \IN$ we have 	$\bar g \bar
        a=g_na_ng_{n-1}a_{n-1}\ldots g_0a_0$  and consequently 
we have analogously to \eqref{eq:F}
      \[F_{\bar g \bar a}(\hat \Psi, S\wr F, n)= F_{\bar g}(\hat \Psi|_{G^F}, S^F, n)\circ l_{\bar a} = l_{\bar a}\circ F_{l_{\bar a^{-1}}(\bar g)}(\hat \Psi|_{G^F}, S^F, n).\]
      With this observation the proof can be carried out in exactly 
      the same way as in the case of a homotopy action. 
  \end{proof}

  \begin{remark}
    The proof of Theorem~\ref{thm:transfer-red_implies_FICwF} 
    given above only uses that $G$ is almost (strongly) transfer reducible
    over $\calf$, not that $G$ is (strongly) transfer reducible. 
    Consequently, Theorem~\ref{thm:transfer-red_implies_FICwF}
    remains true if we replace the assumption 
    ``(strongly) transfer reducible''    
    by the weaker assumption ``almost (strongly) transfer reducible''.
  \end{remark}

%====================================================================

\section{The Farrell-Jones Conjecture with wreath products}
 \label{sec:FJ-wreath}

  \begin{definition}
    \label{def:FJwF}
    A group $G$ is said to satisfy the $L$-theoretic 
    Farrell-Jones Conjecture \emph{with wreath products}
    with respect to the family $\calf$
    if for any finite group $F$ the wreath product
    $G \wr F$ satisfies the $L$-theoretic
    Farrell-Jones Conjecture with respect to the 
    family $\calf$ 
    (in the sense of Definition~\ref{def:L-theoretic_Farrell-Jones_Conjecture}).
    
    If the family $\calf$ is not mentioned, it is by default the family $\VCyc$ of
    virtually cyclic subgroups.

    There are similar versions with wreath products 
    of the $K$-theoretic Farrell-Jones 
    Conjecture and the $K$-theoretic Farrell-Jones 
    Conjecture up to dimension $1$.
  \end{definition}
  
  The FJC with wreath products has first 
  been used in~\cite{Farrell-Roushon(2000)} to deal with finite extensions, 
  see also~\cite[Definition~2.1]{Roushon(2008FJJ3)}.

  \begin{remark} \label{rem:inheritance-FJCwreath}
    The inheritance properties of the FJC 
    for direct products, subgroups, exact sequences and directed
    colimits hold also for the FJC with wreath products
    and can be deduced from the corresponding properties 
    of the FJC itself.
    See for example~\cite[Lemma~3.2, 3.15, 3.16, Satz~3.5]{Kuehl(2009)}.

    For a group $G$ and two finite groups $F_1$ and $F_2$ we have
    $(H \wr F_1 )\wr F_2 \subset H \wr (F_1 \wr F_2)$ and 
    $F_1 \wr F_2$ is finite.
    In particular, if $G$ satisfies the FJC
    with wreath products, then the same is true for any wreath product
    $G \wr F$ with $F$ finite.

    The main advantage of the FJC with
    wreath products is that in addition it passes to overgroups
    of finite index.
    Let $G'$ be an overgroup of $G$ of finite index, i.e., 
    $G \subset G'$, $[G':G] < \infty$. Let $S$ denote a system of representatives of the cosets $G'/G$.
    Then $N:=\bigcap_{s\in S} sGs^{-1}$ is a finite index, normal subgroup of $G'$. 
    Now $G'$ can be embedded in $N \wr G'/N$ (see~\cite[Section~2.6]
    {Dixon-Mortimer(1996)}, \cite[Section~2]{Farrell-Roushon(2000)}). 
    This implies that $G'$ satisfies the FJC with wreath products,
    because $N \wr G'/N$ does by the inheritance properties
    discussed before. 
  \end{remark}

  \begin{theorem} \label{prop:L-FICwF-for-GL_nZ}
    The $L$-theoretic FJC with wreath products
    holds for $\GL_n(\IZ)$.
  \end{theorem}

  \begin{proof}
    We proceed by induction over $n$.
    As $\GL_1(\IZ)$ is finite, the induction beginning is trivial.

    Let $F$ be a finite group. 
    Since $\GL_n(\IZ)$ is transfer reducible over $\calf_n$ by 
    Theorem~\ref{the:transfer-reducible} it follows 
    from Theorem~\ref{thm:transfer-red_implies_FICwF}~\ref{thm:transfer-red_implies_FICwF:normal} that
    $\GL_n(\IZ) \wr F$ satisfies the  $L$-theoretic FJC 
    with respect to $(\calf_n)^\wr$.
    It remains to replace $(\calf_n)^\wr$ by the family $\VCyc$.
    By the Transitivity Principle~\ref{prop:trans-princ} 
    it suffices to prove the $L$-theoretic FJC
    (with respect to $\VCyc$)
    for all groups $H \in (\calf_n)^\wr$.

    Because the FJC with wreath products
    passes to products and finite index overgroups, see
    Remark~\ref{rem:inheritance-FJCwreath} 
    it suffices to consider $H \in \calf_n$.
   
    Combining the induction assumption for $\GL_k(\IZ)$, $k < n$
    with the inheritance properties for direct products, 
    exact sequences of groups and subgroups 
    (see Remark~\ref{rem:inheritance-FJCwreath})
    it is easy to reduce the  FJC
    with wreath products for members of $\calf_n$ to
    the class of virtually poly-cyclic groups.
    Wreath products of virtually poly-cyclic groups
    with finite groups are again virtually poly-cyclic.
    Thus the result follows in this case 
    from~\cite{Bartels-Farrell-Lueck(2011cocomlat)}.    
  \end{proof}
  
  Because for finite $F$ 
  the wreath product $\GL_n(\IZ) \wr F$ can be embedded into
  $\GL_m(\IZ)$ for some $m > n$, there is really no difference
  between the FJC  and the FJC with wreath products for the 
  collection of groups $\GL_n(\IZ)$, $n \in \IN$. 
  Nevertheless, as discussed in the introduction, 
  for $L$-theory the induction only works for
  the FJC with wreath products.

\begin{remark}\label{rem:hyperbolic_groups}
   We also conclude that a hyperbolic group $G$ satisfies the $K$- and
   $L$-theoretic FJC with wreath products.  
   A hyperbolic group is strongly transfer reducible over $\VCyc$ by 
   \cite[Example~3.2]{Wegner(2012FJ_CAT0)} and in particular 
   transfer reducible over $\VCyc$. Hence it satisfies the K-theoretic 
   FJC in all dimensions and the L-theoretic FJC with respect to the family 
   $\VCyc^\wr$ by 
   Theorem~\ref{thm:transfer-red_implies_FICwF}~\ref{thm:transfer-red_implies_FICwF:strongly}.
   By the transitivity principle~\ref{prop:trans-princ} it suffices to show the
   FJC for all groups from $\VCyc^\wr$.  Since those groups are virtually
   polycyclic the FJC holds for them 
   by~\cite[Theorem~0.1]{Bartels-Farrell-Lueck(2011cocomlat)}.
   
   Notice that a group is hyperbolic if a subgroup of finite index is
   hyperbolic. 
   Nevertheless, it is desirable to have the wreath product version
   for hyperbolic groups also since it inherits to colimits of 
   hyperbolic groups
   and many constructions of groups with exotic properties occur as colimits of
   hyperbolic groups.
\end{remark}

%================================================================

\section{Proof of the General Theorem}
\label{sec:Proof_of_the_General_Theorem}

\begin{lemma}\label{lem:from_Z_to_ring_of_integers}
  Let $R$ be a ring whose underlying abelian group  is finitely generated.
  Then both the $K$-theoretic and the $L$-theoretic FJC
  hold for $\GL_n(R)$ and $\SL_n(R)$.
\end{lemma}

\begin{proof} 
  Since the FJC passes to subgroups 
  (by~\cite{Bartels-Reich(2007coeff)}) 
  we only need to treat $\GL_n(R)$. 
  Choose an isomorphism of abelian groups 
  $h' \colon R^n \xrightarrow{\cong} \IZ^k \times T$ 
  for  some natural number $k$ and a finite abelian group $T$. We
  obtain an injection of groups
  \[
  \GL_n(R) \cong \aut_R(R^n) \xrightarrow{f} \aut_{\IZ}(R^n) \xrightarrow{h}
  \aut_{\IZ}(\IZ^{k} \times T),
  \]
  where $f$ is the forgetful map and $h$ comes by conjugation with
  $h'$. 
  Since the FJC  passes to subgroups, it
  suffices to prove the FJC  for $\aut_{\IZ}(\IZ^{k} \times T)$.  
  There is an obvious exact sequence of groups
  \[1 \to \hom_{\IZ}(\IZ^k,T) \to \aut_{\IZ}(\IZ^{k} \times T) 
                        \to \GL_k(\IZ) \times \aut_{\IZ}(T) \to 1.
  \]
  Since $\hom_{\IZ}(\IZ^k,T) $ and $\aut_{\IZ}(T)$ are finite and 
  $\GL_k(\IZ)$  satisfies the FJC, 
  Lemma~\ref{lem:from_Z_to_ring_of_integers} follows 
  from~\cite[Corollary~1.12]{Bartels-Farrell-Lueck(2011cocomlat)}.
\end{proof}

\begin{proof}[Proof of General Theorem] Let $G$ be a
  group which is commensurable to a subgroup $H \subseteq \GL_n(R)$ for some
  natural number $n$.  We have to show that $G$ satisfies the FJC 
  with  wreath products.  
  We have explained in Remark~\ref{rem:inheritance-FJCwreath}
  that the FJC with wreath products 
  passes to overgroups of finite
  index and all subgroups.  
  Therefore it suffices to show that the
  FJC with wreath products holds for $\GL_n(R)$.
  
  Consider a finite group $H$. 
  Let $R \wr H$ be the twisted group ring of $H$
  with coefficients in $\prod_H R$ where the $H$-action on $\prod_H R$ is given
  by permuting the factors. 
  Since $R$ is finitely generated  as abelian group by assumption, 
  the same is true for $R \wr H$. 
  Hence $\GL_n(R \wr H)$ satisfies
  the FJC by Lemma~\ref{lem:from_Z_to_ring_of_integers}.

  There is an obvious injective group homomorphism 
  $\GL_n(R) \wr H=GL_n(R)^H \rtimes H \to \GL_n(R\wr H)$.
  It extends the obvious group monomorphism $\GL_n(R)^H = \GL_n(R^H) \to \GL_n(R \wr H)$ via
  $H \to GL_n( R \wr H)$, $h\mapsto h\cdot I_n$.
  Since the FJC passes to subgroups it holds for $\GL_n(R) \wr H$.

\end{proof}

%%%%%%%%%%%%%%%%%%%%%%%%%%%%%%%%%  References %%%%%%%%%%%%%%%%%%%%%%%%%%

%\bibliographystyle{abbrv}

%\bibliography{dbpub,dbpre,dbpubbuf}
%\bibliography{dbpub,dbpre,dbSLnZ}

\end{document}